\documentclass[11pt]{article}
\usepackage[utf8]{inputenc}
\usepackage{float,caption,setspace,amssymb,amsmath,comment,amsfonts,amsthm,graphicx,wrapfig,booktabs,mathtools,hyperref}
\usepackage[shortlabels]{enumitem}
\usepackage{tikz-cd}
\usepackage{slashed}
\usepackage[english]{babel}
\usepackage{booktabs}
\usepackage{enumitem}

\newcommand{\co}{\mathfrak{c}}

\newcommand{\re}{\mathbb{R}}

\newcommand{\com}{\mathbb{C}}

\newcommand{\bege}{\begin{equation}}
\newcommand{\enge}{\end{equation}}

\newcommand{\A}{{\mathcal{A}}}
\newcommand{\B}{{\mathcal{B}}}

\newcommand{\End}{\text{End}}

\newcommand{\beq}{\begin{eqnarray}}
\newcommand{\eeq}{\end{eqnarray}}

\newcommand{\benu}{\begin{enumerate}}
\newcommand{\enu}{\end{enumerate}}
\newcommand{\cs}{\Gamma_c^{\infty}(E)}
\newcommand{\DD}{D_{\bullet}}
\newcommand{\DF}{\{D_x\}_{x\in X}}

\newcommand{\CS}{C^*}
\newcommand{\WHD}{\widehat{D}}
\newcommand{\gi}{\Gamma^{\infty}(E)^G}

\newcommand{\bN}{\mathbb{N}}
\newcommand{\bZ}{\mathbb{Z}}

\newcommand{\begeq}[1]{\begin{equation} \label{#1}}
\newcommand{\eneq}{\end{equation}}

\DeclareMathOperator{\SF}{SF}
\DeclareMathOperator{\Dom}{Dom}

\DeclareMathOperator{\ee}{\mathcal{E}}

\DeclareMathOperator{\Index}{Index}
\DeclareMathOperator{\Ind}{Ind}
\DeclareMathOperator{\Sf}{sf}
\DeclareMathOperator{\loc}{loc}
\DeclareMathOperator{\tr}{tr}

\theoremstyle{definition} 

\newtheorem{definition}{Definition}
\newtheorem{example}[definition]{Example}
\newtheorem{remark}[definition]{Remark}

\theoremstyle{plain} 

\newtheorem{lemma}[definition]{Lemma}
\newtheorem{corollary}[definition]{Corollary}
\newtheorem{proposition}[definition]{Proposition}
\newtheorem{theorem}[definition]{Theorem}

\numberwithin{equation}{section}
\numberwithin{definition}{section}

\title{Equivariant Spectral Flow for Families of Dirac-type Operators}
\author{Peter Hochs\footnote{p.hochs@math.ru.nl}\;   and Aquerman Yanes\footnote{a.k.yanes@math.ru.nl, aquermanym@gmail.com}\\ Institute for Mathematics, Astrophysics and Particle Physics\\ Radboud University}

\begin{document}

\maketitle

\begin{abstract}
In the setting of a proper, cocompact action by a locally compact, unimodular group $G$ on a Riemannian manifold, we construct equivariant spectral flow of paths of Dirac-type operators. This takes values in the $K$-theory of the group $C^*$-algebra of $G$. In the case where $G$ is the fundamental group of a compact manifold, the summation map maps equivariant spectral flow on the universal cover to classical spectral flow on the base manifold. We obtain ``index equals spectral flow'' results. In the setting of a smooth path of $G$-invariant Riemannian metrics on a $G$-spin manifold, we show that the equivariant spectral flow of the corresponding path of spin Dirac operators relates delocalised $\eta$-invariants and $\rho$-invariants for different positive scalar curvature metrics to each other.
\end{abstract}

\tableofcontents

\section{Introduction}

Let $M$ be a compact Riemannian manifold and $E\to M$ a Hermitian vector bundle equipped with a Dirac-type operator $D\colon \Gamma^{\infty}(M,E)\to L^2(M,E)$. It is well-known that $D$ is an essentially self-adjoint operator and that it satisfies a Lichnerowicz-Weitzenböck formula 
\begin{equation*}
	D^2=\Delta+\kappa,
\end{equation*}
where $\Delta$ is a Laplace-type operator and $\kappa$ is a self-adjoint endomorphism of $E$ built out of the curvature. Since $D$ is an elliptic operator on a compact manifold,  it is Fredholm, and hence it has an index
\begin{equation*}
	\Index(D)\in \mathbb{Z}.
\end{equation*}
It is well-known that there exists a topological formula one can use in order to compute this index (\cite{atiyahsingercompact}), and that it vanishes when the dimension of $M$ is odd. One may consider, instead, in the odd-dimensional case, a suitably continuous family of Dirac-type operators $\{D_t\colon\Gamma^{\infty}(M,E)\to L^2(M,E)\}_{t\in \re}$, parametrised by the real line. The so-called Dirac--Schrödinger operator $-i\partial_t-i\DD$ on the Hilbert space $L^2(M\times \re,E\boxtimes \com)$, which is an elliptic differential operator on the (even dimensional) manifold $M\times\re$, can then be considered. If the family $\{D_t\}_{t\in \re}$ is suitably invertible outside a compact interval $K\subset \re$, then $-i\partial_t-i\DD$ is Fredholm (\cite{MR0569070, anghelcallias, robbinsalamon}). 

It is possible to associate an integer, which we call \textit{classical spectral flow}, to the family $\{D_t\}_{t\in\re}$, which represents the net number of positive eigenvalues of $D_t$ changing sign as $t$ varies through $\re$ (see e.g.: \cite{apsiii, robbinsalamon, phil, wahl, MR4605812}). We denote this number by
\begin{equation*}
	\Sf(\{D_t\}_{t\in\re})\in \mathbb{Z}.
\end{equation*}
It turns out that the Fredholm index of the Dirac--Schrödinger operator $-i\partial_t-i\DD$ on $L^2(M\times \re,E\boxtimes \com)$ coincides with the classical spectral flow of the family $\{D_t\}_{t\in \re}$ (\cite{robbinsalamon, wahl, AW11}), that is,

\begin{equation}\label{indexeqsf}
	\Index(-i\partial_t-i\DD)=\Sf(\{D_t\}_{t\in\re})\in\mathbb{Z}.
\end{equation}

In the context of $K\!K$-theory, equation (\ref{indexeqsf}) can be generalised as follows: let $\DF$ be a family of Dirac-type operators parametrised by a connected, complete Riemannian manifold $X$, and let $P_X$ be a symmetric, elliptic, first-order differential operator on a Hermitian vector bundle $F \to X$. If $P_X$ has finite propagation speed, then it is well-known that it defines a class in $K$-homology (\cite[\S 10.6]{higsonroe}):
\begin{equation*}
	[P_X]\in K\!K^1(C_0(X),\com).
\end{equation*}
 The family $\DF$ defines a regular self-adjoint and Fredholm operator $\DD$ on $C_0(\re,L^2(M,E))$, which in turn defines a class 
 \begin{equation*}
	\SF(\DD)\in K\!K^1(\com,C_0(X)),
\end{equation*}
 called \textit{$K\!K$-theoretic spectral flow} (\cite{kl1,koen}). Under suitable conditions, the operator $P_X-i\DD$ on $L^2(M,E)\otimes L^2(X, F)$ can be proven to be Fredholm, and its index given by
\begin{equation}\label{eqtheindthe}
	\Index(P_X-i\DD)=\SF(\DD)\otimes_{C_0(X)}[P_X]\in K\!K^0(\com,\com)\simeq \mathbb{Z},
\end{equation}
so that the right-hand of (\ref{eqtheindthe}) side can be perceived as a $K\!K$-theoretic approach to defining spectral flow. 

We note here that these constructions are dependent on the fact that the manifold $M$ is \textit{compact}. In this paper, we consider a possibly non-compact Riemannian manifold $M$ equipped with a proper, isometric and cocompact action by a unimodular, locally compact group $G$. This setup allows us to define an equivariant version of spectral flow for a  family of $G$-equivariant Dirac-type operators on $M$, living in the $K$-theory of the (maximal or reduced) group $C^*$-algebra $\CS(G)$ (Definition \ref{defgspecflow}). 
In our construction, we build on the $K\!K$-theoretic framework of spectral flow in the general context of operators on Hilbert $C^*$-modules \cite{koen, kl1}. We also adapt and apply techniques from  \cite{haoguo} related to Hilbert $C^*$-modules associated to proper group actions.

Other notions of equivariant spectral flow were constructed and applied  in \cite{IJW21, LW22} for compact groups, \cite{LP03} for fundamental groups of compact manifolds acting on their universal covers, and \cite{MR2200374} for families of equivariant Toeplitz operators. A higher version of spectral flow for families of operators was constructed in  \cite{DZ96, DZ98}. This notion was employed in \cite{MR4316135} in order to define an equivariant notion of spectral flow in the context of actions by compact Lie groups. A von Neumann-algebraic version of spectral flow was studied in \cite{AW11}. A general definition of noncommutative spectral flow for paths of operators on Hilbert $C^*$-modules was considered in \cite{wahl}.

This paper is structured as follows: we start in Section \ref{secpreliminaries} by laying out the preliminaries and basic constructions used throughout the paper. In Section \ref{sec results} we discuss and present all the main results of the paper, while the following sections are devoted to laying out the proofs of these results. More precisely, we show in Section \ref{secgspectralflow} that equivariant spectral flow for a path of Dirac-type operators is well-defined under Assumptions \ref{assump1}-\ref{assump3} (Definition \ref{defgspecflow}), and that it refines classical spectral flow under an integration map (Proposition \ref{intgspecfloeq}). In Section \ref{seceqindextheory}, we prove our ``index equals spectral flow" results (Theorem \ref{propfredindeqsfg1} and \ref{propindtheorem}). Finally, in Section \ref{seclocind} we prove how our equivariant spectral flow is related to the notion of delocalised eta invariants and higher rho invariants (Theorems \ref{thm eta} and \ref{theoremmain}).

\subsection*{Acknowledgements}

We thank Koen van den Dungen, Hao Guo, Bram Mesland and Adam Rennie for helpful discussions in person and via email.

\section{Preliminaries} \label{secpreliminaries}

 In this section we collect the main definitions and constructions concerning the Hilbert $\CS(G)$-module counterpart of the $L^2$-space of compactly supported sections on a $G$-manifold $M$, due to Kasparov \cite{kasparov}. We follow the work of \cite{haoguo}, where a generalisation of this Hilbert module was made in order to encompass Sobolev-type Hilbert modules. 
 
 We also recall the definitions of secondary invariants that we will prove to be related by spectral flow: delocalised $\eta$-invariants and $\rho$-invariants.

Throughout this paper, we consider  a Riemannian manifold $M$ and a unimodular, locally compact group $G$ acting properly and isometrically on $M$. Let $d\mu$ be the Riemannian density on $M$, and fix a Haar measure $dg$ for $G$. Let $E \to M$ be a $G$-equivariant Hermitian vector bundle, where $G$ acts via isometries between the fibers. We also fix a $G$-equivariant, symmetric, elliptic, first-order differential operator $D\colon \cs\to \cs$.

\subsection{$G$-Sobolev Modules}\label{subsecgsobolevmod}

 We denote the $L^2$-inner product of two compactly supported smooth sections $s_1,s_2\in\cs$ by
 \begin{equation}\label{l2prod}
	\langle s_1,s_2\rangle_{L^2(E)}:=\int_M (s_1(m),s_2(m))_{E_m}\,d\mu(m).
\end{equation}
 The vector space completion of $\cs$ with respect to the norm defined by (\ref{l2prod}) is denoted by $L^2(E)$. The completion of $\cs$ with respect to the graph norm defined by $D$ defines a first Sobolev space $W^1(E)$. Similarly, for each $k\in \mathbb{Z}_+$ the graph norm of $D^k$ defines a $k$-th Sobolev space, which we denote by $W^k(E)$.

 We use the symbol $\CS(G)$ in order to denote either one of the following two group $C^*$-algebras, which are completions of the space of compactly supported function $C_c(G)$: the \textit{maximal group $C^*$-algebra}, denoted by $\CS_{\rm{max}}(G)$, and the \textit{reduced group $C^*$-algebra}, denoted by $\CS_r(G)$. Their norm are related by the inequalities
\begin{equation}\label{l1ineq}
\|f\|_{\CS_r(G)}\leq\|f\|_{\CS_{\text{max}}(G)}\leq \|f\|_{L^1(G)},
\end{equation} for every $f\in C_c(G)$. We write $\|f\|_{\CS(G)}$ to denote either the maximal or reduced norm, if no confusion arises.

The following construction, due to Kasparov \cite{kasparov}, is a Hilbert $\CS(G)$-module analogue of the Hilbert space $L^2(E)$. The group $G$ acts naturally on compactly supported smooth sections $s\in \cs$ by 
\begin{equation*}
	g(s)(m):=g(s(g^{-1}m)),
\end{equation*}
for each $g\in G$ and $m\in M$. Let $s_1,s_2\in \cs$ and $f\in C_c(G)$ a compactly supported function, and define
\begin{equation}\label{eq e0 inner prod}
\begin{split}
\langle s_1,s_2\rangle_{\ee^0(E)}&:=\left(g \mapsto \langle s_1,g(s_2)\rangle_{L^2(E)}\right)\in C_c(G),\\
	(s_1\cdot f)&:=\left(m\mapsto \int_G g(s_1)(m)f(g^{-1})\,dg\right)\in \cs.	
\end{split}
\end{equation}
Note that the inner product $\langle s_1,s_2\rangle_{\ee^0(E)}$ is indeed a compactly supported function of $G$, by properness of the action. Define the norm
\begin{equation}\label{eq e0 norm}
	\|s\|_{\ee^0(E)}:=\|\langle s,s\rangle_{\ee^0(E)}\|_{\CS(G)}^{1/2}.
\end{equation}
The vector space completion of $\cs$ with respect to the norm (\ref{eq e0 norm}) is denoted by $\ee^0(E)$. By taking completions, one produces a $\CS(G)$-valued inner product and $\CS(G)$-action on the right, which turns $\ee^0(E)$ into a Hilbert $\CS(G)$-module. This construction is generalised in \cite{haoguo}: define for each $k\in\mathbb{Z}_+$ and $s_1,s_2\in \cs$ the $C_c(G)$-valued inner product
\begin{equation*}
		\langle s_1,s_2\rangle_{\ee^k(E)}:=\displaystyle\sum_{l=0}^k \langle D^ls_1,D^ls_2\rangle_{\ee^0(E)},
\end{equation*}
which is a Sobolev-type generalisation of the $\ee^0$-inner product in (\ref{eq e0 inner prod}). Then, define the norm
\begin{equation}\label{eqeknorm}
	\|s\|_{\ee^k(E)}:=\|\langle s,s\rangle_{\ee^k(E)}\|_{\CS(G)}^{1/2}.
\end{equation}
The vector space completion of $\cs$ with respect to the norm defined in (\ref{eqeknorm}) is called the \textit{$k$-th $G$-Sobolev module} defined with respect to $D$, and is denoted by $\ee^k(E)$. Positivity of these inner products is proved in  \cite[Lemma 3.2]{haoguo}.

The operator $D^k$ defines a map $D^k\colon \cs\to \ee^0(E)$, for each $k\in\mathbb{Z}_+$. Since $D$ is $G$-equivariant, we see that $D^k$ is symmetric with respect to the $\ee^0$-inner product in (\ref{eq e0 inner prod}), and thus the map $D^k\colon \cs\to \ee^0(E)$ is  closable. We denote the closure operator by $\overline{D^k}\colon \Dom(\overline{D^k})\to \ee^0(E)$ and note that the graph norm of $\overline{D^k}$ on $\Dom(\overline{D^k})$ is equivalent to the norm (\ref{eqeknorm}), and thus $\Dom(\overline{D^k})=\ee^k(E)$. Hence $\overline{D^k}\colon \ee^k(E)\to \ee^0(E)
$ is a bounded operator for each $k\in\mathbb{Z}_+$. We omit the overline notation for the closure operator in what follows, so that given a $G$-equivariant, symmetric, elliptic, first-order differential operator $D$ on a Hermitian $G$-equivariant vector bundle $E\to M$, we may talk about the bounded operator
 \begin{equation}\label{eq dk}
 	D^k\colon\ee^k(E)\to \ee^0(E).
 \end{equation}
 
 Recall the following important property of unbounded operators between Hilbert $C^*$-modules (see \cite[\S 9]{lance} for further details):
 \begin{definition}\label{def regular}
	A \textit{regular} operator between Hilbert $B$-modules $\ee, \mathcal{F}$ is a densely defined closed $B$-linear map $T\colon \Dom(T)\subset \ee\to \mathcal{F}$ such that $T^*\colon \Dom(T^*)\to \mathcal{F}$ is densely defined and $1+T^*T$ has dense range.
\end{definition}
Adjointability of (\ref{eq dk}) is proved in \cite[Proposition 3.8]{haoguo}. We remark that if $D$ has finite propagation speed, then the operator $D$ on $\ee^0(E)$, of particular interest for us, is regular self-adjoint on the domain $\ee^1(E)$ by \cite[Proposition 5.5]{haoguo}. (See \cite[Theorem 5.8]{kasparov} for a proof in the cocompact case, without the finite propagation speed assumption.)

\subsection{The Fredholm Index}\label{subsecfredind}

	In this subsection, we lay out the notion of \textit{Fredholm operators} on a Hilbert $C^*$-module, and recall how to define their index. More details on $K$-theory and (unbounded) $K\!K$-theory and its role in index theory can be found in the standard reference \cite{blackadar}. See also \cite[\S2]{koen} for a concise summary of the necessary concepts. 
	
Let $B$ be a $\sigma$-unital $C^*$-algebra and let $\ee$ be a countably generated Hilbert $B$-module. An adjointable operator $T\in \mathcal{L}(\ee)$ is called \textit{Fredholm} if there exists $Q\in \mathcal{L}(\ee)$ such that $TQ-1, QT-1\in \mathcal{K}(\ee)$, that is, if $T$ is invertible modulo compact operators in $\mathcal{L}(\ee)$. We now recall how to define the index of such operators as an element in the $K$-theory group $K_0(B)$, following \cite{MR1045144}. 

Let $\mathcal{H}_B$ denote the standard $B$-Hilbert module (see e.g.: \cite[Example 15.1.7]{MR1222415}), and let $\mathcal{F}(\ee)$ denote the space of adjointable Fredholm operators on $\ee$. By Kasparov stabilisation, we have $\ee\oplus \mathcal{H}_B\simeq \mathcal{H}_B$, which yields an embedding $\mathcal{F}(\ee)\hookrightarrow \mathcal{F}(\mathcal{H}_B)$ via $T\mapsto T\oplus 1$. Let $\pi\colon\mathcal{L}(\mathcal{H}_B)\to \mathcal{L}(\mathcal{H}_B)/\mathcal{K}(\mathcal{H}_B)$ denote the projection to the Calkin algebra. Then, for each $T\in \mathcal{F}(\ee)$ we get an invertible element $\pi(T\oplus 1)\in\mathcal{L}(\mathcal{H}_B)/\mathcal{K}(\mathcal{H}_B)$, which defines an odd $K$-theory class $[\pi(T\oplus 1)]\in K_1(\mathcal{L}(\mathcal{H}_B)/\mathcal{K}(\mathcal{H}_B))$. The short exact sequence
\begin{equation*}
	0\longrightarrow \mathcal{K}(\mathcal{H}_B)\longrightarrow \mathcal{L}(\mathcal{H}_B)\longrightarrow \mathcal{L}(\mathcal{H}_B)/\mathcal{K}(\mathcal{H}_B)\longrightarrow 0
\end{equation*}
yields a 6-term exact sequence in $K$-theory. Since $K_0(\mathcal{L}(\mathcal{H}_B))=K_1(\mathcal{L}(\mathcal{H}_B))=0$, the boundary map
\begin{equation*}
	\partial\colon K_1(\mathcal{L}(\mathcal{H}_B)/\mathcal{K}(\mathcal{H}_B))\to K_0(\mathcal{K}(\mathcal{H}_B))\simeq K_0(B)
\end{equation*}
defines an isomorphism between the $K$-theory groups. The \textit{Fredholm index} of $T \in \mathcal{F}(\ee)$ is then defined by
\begin{equation}\label{eqindexdef}
	\Index(T):=\partial[\pi(T\oplus 1)]\in K_0(B).
\end{equation}

Let now $T\colon\Dom(T)\subset \ee\to \ee$ be a regular operator on the Hilbert $B$-module $\ee$. A \textit{right parametrix} for $T$ is an operator $Q_R\in\mathcal{L}(\ee)$ such that the operator $TQ_R$ is adjointable and $TQ_R-1$ is compact. A \textit{left parametrix} for $T$ is an operator $Q_L\in\mathcal{L}(\ee)$ such that $Q_LT$ is closable, $\overline{Q_LT}$ is adjointable and $\overline{Q_LT}-1$ is compact. A \textit{parametrix} for $T$ is an operator $Q\in\mathcal{L}(\ee)$ which is both a left and right parametrix. The regular operator $T$ is called \emph{Fredholm} if it has a left and a right parametrix. 
 Equivalently, $T$ is Fredholm if its bounded transform $T(1+T^*T)^{-1/2}\in \mathcal{L}(\ee)$ is Fredholm (\cite[Lemma 2.2]{joachim}). The index of a regular Fredholm operator $T\colon\Dom(T)\to\ee$ is then defined to be the index (\ref{eqindexdef}) of its bounded transform. 
 
  \begin{definition}[\rm{\cite[Definition 2.3]{wahl}}]
  	Let $\ee$ be a Hilbert $B$-module. A \textit{normalising function} for a regular self-adjoint operator $T\colon\Dom(T)\to\ee$ is an odd, non-decreasing, smooth function $\chi\colon\re\to\re$ such that $\chi(0)=0$, $\chi'(0)>0$, $\lim_{x\to\infty}\chi(x)=1$ and $\chi(T)^2-1\in\mathcal{K}(\ee)$.
  \end{definition}  
 If $T$ is regular self-adjoint and Fredholm, then it admits a normalising function $\chi\in C^{\infty}(\re)$ (\cite[Proposition 2.14]{koen}). It follows that the regular self-adjoint operator $\chi(T)$ defines a class
  \begin{equation}\label{eq kktheory class}
 	[T]:=[(\ee,\chi(T),1_\com)]\in K\!K^{j}(\com,B),
 \end{equation}
where $j=0$ if $\ee$ is a $\mathbb{Z}_2$-graded Hilbert module and $T$ is an odd operator, and $j=1$ if these structures are ungraded. The class (\ref{eq kktheory class}) is independent of the choice of normalising function, by \cite[Proposition 2.14]{koen}. If $T$ has compact resolvents, then the function $\chi(x)=x(x^2+1)^{-1/2}$ is a normalising function for $T$. The class (\ref{eq kktheory class}) then corresponds to the bounded transform of the unbounded $(\com,B)$-cycle $(\ee,T,1_\com)$. 

Let 
 \begin{equation*}
 	T=\begin{pmatrix}
 		0 &T_-\\T_+&0
 	\end{pmatrix}
 \end{equation*}
 be an odd Fredholm operator on the standard $\mathbb{Z}_2$-graded Hilbert module $\mathcal{H}_B\oplus \mathcal{H}_B$. Throughout this paper, we fix the identification $$KK^0(\com,B)\simeq K_0(B)$$ mapping the class $[T]\in KK^0(\com,B)$ to $\Index(T_+)\in K_0(B)$ (cf. \cite[\S2.1]{wahl}).

 \subsection{Delocalised $\eta$-Invariants and $\rho$-Invariants}
 
 Classical spectral flow relates $\eta$-invariants of different Dirac operators to each other (\cite{apsiii}). We will show that equivariant spectral flow yields  analogous relations between``higher'' analogues of $\eta$-invariants: delocalised $\eta$-invariants and $\rho$-invariants. (See Theorems \ref{thm eta} and  \ref{theoremmain}).

 In the setting of proper, cocompact group actions, the $\eta$-invariant \cite{APSI} was generalised to the \emph{delocalised $\eta$-invariant} \cite{HWW23, Lott99}. Recall that every proper $G$-manifold $M$ admits a smooth cutoff function $\co\colon M\to[0,1]$ such that, for every $m\in M$,
\begin{equation}\label{eq cutoff c}
	\int_G \co(g^{-1}m)^2\,dg =1.
\end{equation}
We do suppose here that $M/G$ is compact. Then the cutoff function $\co \in C^{\infty}_c(M)$ can be chosen to be compactly supported. Let $\kappa_t$ be the Schwartz kernel of the operator $De^{-tD^2}$. Fix an element $h \in G$ whose centraliser $Z$ is unimodular. Then there is a $G$-invariant measure $d(xZ)$ on $G/Z$.
\begin{definition}
If the following converges, then the \emph{delocalised $\eta$-invariant} of $D$ at $h$ is
\begeq{eq def eta g}
\eta_h(D) := \frac{2}{\sqrt{\pi}} \int_0^{\infty} \int_{G/Z} \int_M
\co(m)^2 \tr\bigl(xhx^{-1} \kappa_{t^2}(xh^{-1}x^{-1}m,m) \bigr)\, dm\, d(xZ)\, dt.
\eneq
\end{definition} 
  
 For the definition of $\rho$-invariants, we  
 suppose, in addition to the assumptions at the start of this section, that $G = \Gamma$ is a finitely generated discrete group, and that it acts freely on $M$. Let $C_0(M)$ denote the $C^*$-algebra of continuous functions on $M$ which vanish at infinity.
We view elements in $C_0(M)$ as operators on $L^2(E) \otimes l^2(\bN)$ by pointwise multiplication on the first factor. 
\begin{definition}\label{def D star}
An operator $T  \in \B(L^2(E) \otimes l^2(\bN))$
\begin{itemize}
\item has \emph{finite propagation} if there is an $R>0$ such that for all $f_1, f_2 \in C_0(M)$ whose supports are at least a distance $R$ apart, we have $f_1 \circ T  \circ f_2 = 0$;
\item is \emph{pseudolocal} if $[T,f]$ is a compact operator for all $f \in C_0(M)$;
\item is \emph{locally compact} if $T \circ f$ and $f \circ T$ are compact operators for all $f \in C_0(M)$.
\end{itemize}
\end{definition}

\begin{definition}[\rm{\cite[p.\ 42]{Roe96}; \cite[Def.\ 5.2.1]{WillettYu}; \cite[\S 1.1]{piazzaschick}}]\label{def c star algs}
The algebra $D^*(M)^{\Gamma}$ is the closure in $\mathcal{B}(L^2(E) \otimes l^2(\mathbb{N}))$ of the subalgebra of $\Gamma$-equivariant, pseudolocal operators with finite propagation. The \emph{equivariant Roe algebra} $C^*(M)^{\Gamma} \subset D^*(M)^{\Gamma}$ is the closure in $\mathcal{B}(L^2(E) \otimes l^2(\mathbb{N}))$ of the subalgebra of $\Gamma$-equivariant, locally compact operators with finite propagation. 
\end{definition}

Now suppose that $M$ is odd-dimensional, and that it has a $\Gamma$-equivariant spin structure, that $E$ is the corresponding spinor bundle (ungraded), and that the Riemannian metric $g^{TM}$ on $M$ has positive scalar curvature. Then, the spin Dirac operator $\slashed{D}$ is invertible by Lichnerowicz' formula.
Let $\delta_1 \in l^2(\mathbb{N})$ be the function 

\begin{equation*}
	\begin{cases}
		\delta_1(1)=1,\\
		\delta_1(n)=0,\;\;n\neq 1.
	\end{cases}
\end{equation*} 
Let 
$p_1\colon l^2(\mathbb{N}) \to l^2(\mathbb{N})$ be the projection given by $p_1(\varphi) = \varphi(1) \delta_1$ for all $\varphi \in l^2(\mathbb{N})$. Let $\chi\colon \re \to [0,1]$ be a normalising function that takes values $\pm 1$ on the spectrum of $\slashed{D}$.  By Lemma 2.1 in \cite{Roe16}, the operator  $\chi(\slashed{D}) \otimes 1_{l^2(\mathbb{N})}$ lies in $D^*(M)^{\Gamma}$. Its square is the identity, so the operator $(\chi(\slashed{D}) \otimes p_1 + 1)/2$ is a projection in $D^*(M)^{\Gamma}$. It is the projection onto the positive spectrum of $\slashed{D}$, so in particular independent of $\chi$.
Hence
it defines a class
\begin{equation}\label{eq Kthry rho}
[(\chi(\slashed{D}) \otimes p_1 + 1)/2] \in K_0(D^*(M)^{\Gamma}).
\end{equation}
\begin{definition}\label{def rho invar}
The \emph{$\rho$-invariant} of the positive scalar curvature metric $g^{TM}$ is the class \eqref{eq Kthry rho}. It is denoted by $\rho(g^{TM})$.
\end{definition}

\begin{remark}
There is in fact a relation between $\rho$-invariants and delocalised $\eta$-invariants, see Theorem 4.3 in \cite{xieyu}. That result applies to a version of the $\rho$-invariant defined in terms of Yu's localisation algebras, but see Section 6 of \cite{xieyu2} for the relation with Definition \ref{def rho invar}.
\end{remark}

 \section{Results}\label{sec results}
 We now lay out the main results of this paper. In Subsection \ref{sec SFG}, we put forward our construction of the equivariant spectral flow for families of Dirac-type operators. In Subsection \ref{subsec gsf as ref}, we prove that equivariant flow refines classical spectral flow under an integration map. We then prove ``index equals spectral flow" theorems in Subection \ref{sec SF index}. Finally, in Subsection \ref{sec rho eta invar}, we make a connection between our equivariant spectral flow to delocalised $\eta$-invariants and $\rho$-invariants.
 \subsection{$G$-Spectral Flow}\label{sec SFG}
 
 Let $M$, $E$ and $G$ be as at the start of Section \ref{secpreliminaries}, and assume that $G$ acts cocompactly on $M$.
 From now on, 
let $X$ be a locally compact, paracompact topological space. We consider a family $\{D_x\colon \cs\to L^2(E)\}_{x\in X}$ of $G$-equivariant, symmetric, elliptic, first-order differential operators on $E \to M$. For each $x\in X$, let $W^1_x(E)$ be the closure of $\cs$ under the graph norm of the operator $$D_x\colon \cs \to L^2(E).$$
 
 We make the following set of assumptions:

\begin{enumerate}[(1)]
	\item\label{assump1}
    	The domains $W^1_x(E)$ are independent of $x\in X$ (here we write $W^1(E):= W^1_{x_0}(E)$ for a fixed $x_0\in X$);
    	\item\label{assump2} The map $D\colon X\to \mathcal{B}(W^1(E),L^2(E))\colon x\mapsto D_x$ is norm-continuous;  	    	\item\label{assump3}
	There exists a compact subset $K\subset X$ such that
	\begin{equation}\label{weitzenbock}
	D_x^2=L_x^*L_x +\kappa_x,\end{equation} for every $x\in X\backslash K$, where $L_x$ is a $G$-equivariant differential operator, and $\kappa_x\in\End(E)^G$ is a $G$-equivariant endomorphism such that $\kappa_x\geq c>0$ fiberwise.\end{enumerate}
\begin{remark}\label{remark domains}
	In Assumption \ref{assump1}, by \textit{independent of $x\in X$} we mean that the Banach space structures of the domains $W^1_x(E)$ are equivalent. In other words, the operators $D_x\colon \cs \to L^2(E)$ have equivalent graph norms, although not necessarily \textit{uniformly} equivalent. (See Definition \ref{def unif graph norm} later on.) We note that if, additionally, $G=\Gamma$ is a discrete group acting freely on $M$, then the assumption that $M/\Gamma$ is compact already entails that the domains $W^1_x(E)$ do not depend on $x\in X$. (See Proposition \ref{prop garding g discrete} later on.)
\end{remark}

\begin{remark}
	In Assumption \ref{assump3}, the differential operator $L_x$ is defined as an operator $ \cs\to\Gamma_c^{\infty}(E')$, where $E'\to M$ is another Hermitian $G$-equivariant vector bundle. This assumption  is reminiscent of Lichnerowicz' formula for generalised Dirac operators. In that case, the operator $L_x$ is given by a $G$-equivariant connection $\nabla_x\colon \cs\to \Gamma_c^{\infty}(E\otimes T^*M)$.
\end{remark}

Let $\ee^1_x(E)$ denote the first $G$-Sobolev module defined with respect to the operator $D_x$, for each $x\in X$. The following is part of Lemma \ref{a1lemma} below.
\begin{lemma}\label{lem E1x}
The norms on the spaces $\ee_x^1(E)$ are equivalent.
\end{lemma}
From now on, we fix $x_0 \in X$ and write $\ee^1(E):= \ee^1_{x_0}(E)$. By Lemma \ref{lem E1x}, this does not depend on the choice of $x_0$, up to bounded isomorphisms of Hilbert $C^*(G)$-modules.

Consider the Hilbert $C_0(X,\CS(G))$-module $C_0(X,\ee^0(E))$ and its subspace 
\begeq{eq dom D dot}
\{\eta\in C_0(X,\ee^0(E)):\eta(x)\in \ee^1(E), \;x \mapsto D_x \eta(x) \in C_0(X,\ee^0(E))\}.
\eneq

Define the unbounded operator $\DD$ on $C_0(X,\ee^0(E))$, with domain \eqref{eq dom D dot}, by
\begin{equation}\label{eqopdd}
	(\DD\eta)(x):=D_x(\eta(x)),
\end{equation}
 for all $\eta$ in \eqref{eq dom D dot}.
\begin{proposition}\label{prop D dot Fred}
The operator $\DD$ is regular self-adjoint and Fredholm.
\end{proposition}
This proposition is proved after Proposition \ref{koen32}.

By Proposition \ref{prop D dot Fred}, the operator $\DD$ defines a class
	\begin{equation}\label{eqclass}
		[\DD]\in K\!K^1(\com,C_0(X,\CS(G))),
	\end{equation}
	by means of (\ref{eq kktheory class}).	
\begin{definition}[$G$-Spectral Flow]\label{defgspecflow}
	Let $\{D_x\colon \cs\to L^2(E)\}_{x\in X}$ be a family of $G$-equivariant, symmetric, elliptic, first-order differential operators on a manifold $M$ equipped with a proper, isometric and cocompact action by a unimodular, locally compact group $G$, and suppose that Assumptions \ref{assump1}-\ref{assump3} are satisfied. We denote by 
	\begin{equation*}
		\text{SF}_G(\DD)\in K\!K^1(\com,C_0(X,\CS(G)))
	\end{equation*}
	the class given by (\ref{eqclass}), and call it the \textit{$K\!K$-theoretic $G$-spectral flow} of the family $\DF$. In the special case $X=\re$, we write
	\begin{equation}\label{gspectralf}
		\Sf_G(\DD)\in K_0(\CS(G)),
	\end{equation}
	to denote the image of $\SF_G(\DD)$ under the Bott periodicity isomorphism $K\!K^1(\com,C_0(\re,\CS(G)))\simeq K_0(\CS(G))$. We call the class (\ref{gspectralf}) the \textit{G-spectral flow} of the family $\{D_x\}_{x\in \re}$.
\end{definition}

\begin{remark}
	We note here that if $X$ is a \emph{compact} topological space, then Assumption \ref{assump3} holds vacuously (take $K=X$), and hence we can extend our definition of ($K\!K$-theoretic) $G$-spectral flow to families of more general $G$-equivariant, symmetric, elliptic, first-order differential operators. 
\end{remark}
	We obtain the following natural property of $G$-spectral flow:
\begin{proposition}\label{prop k equals 0}
		 Suppose that the endomorphism $\kappa_x$ in \eqref{weitzenbock} satisfies $\kappa_x \geq c>0$ for all $x \in X$, i.e.\ that we can take $K = \emptyset$ in Assumption (3). 
Then $\SF_G(\DD)=0$.
	\end{proposition}

We prove Proposition \ref{prop k equals 0} at the end of Subsection \ref{sec pf SFG}.

\begin{remark}\label{remark unb sf}
	Let $\psi\in C_0(X)$ be strictly positive, and equal to $1$ on the compact $K$ of Assumption \ref{assump3}. In Proposition \ref{a123} we show that $(C_0(X,\ee^0(E)),\psi^{-1}\DD,1_{\com})$ is an odd unbounded Kasparov $(\com,C_0(X,\CS(G)))$-cycle. The bounded transform class
	\begin{equation}\label{eqkkclass}
		[\psi^{-1}\DD]:=[(C_0(X,\ee^0(E)),\psi^{-1}\DD((\psi^{-1}\DD)^2+1)^{-1/2},1_{\com})]
	\end{equation}
	in $K\!K^1(\com,C_0(X,\CS(G)))$ is then well-defined. Following \cite[Proposition 3.7]{koen}, one can see that (\ref{eqclass}) and (\ref{eqkkclass}) actually coincide, and hence (\ref{eqkkclass}) is independent on the choice of function $\psi\in C_0(X)$. It follows that $K\!K$-theoretic spectral flow can also be realised as
	\begin{equation*}
		\SF_G(\DD)=[\psi^{-1}\DD]\in K\!K^1(\com,C_0(X,\CS(G))).
	\end{equation*}
\end{remark}
\begin{remark}\label{remark wahl}
	Under the condition that the family $\{D_x\colon \ee^1(E)\to \ee^0(E)\}_{x\in X}$ has \textit{locally trivialiasing families} (see \cite[Definition 3.9]{wahl}), we note that the non-commutative spectral flow developed in \cite{wahl} can be employed in order to define an equivariant version of spectral flow in our context. We point out that, by \cite[Proposition 2.21]{koen}, this notion coincides with the one we propose in Definition \ref{defgspecflow}. Moreover, even though the non-commutative spectral flow in \cite{wahl} depends on the choice of so-called spectral sections, since the operators in our family are invertible outside the compact set $K\subset X$, one can take these spectral sections to be projection onto the positive spectrum, yielding a canonical definition of the spectral flow in this setting.
\end{remark}

\begin{example}[Classical spectral flow]\label{excpt}
	Let $M$ be a compact Riemannian manifold and $E\to M$ a Dirac bundle. Let $\{D_x\colon \text{Dom}(D_x)\to L^2(E)\}_{x\in X}$ be a norm-continuous family of (closures of) \textit{Dirac-type} operators on $L^2(E)$. That means that there exists, for each $x\in X$, a Clifford action $c_x\colon T^*M\to \End(E)$ and Clifford connection $\nabla_x\colon \Gamma^{\infty}(E)\to \Gamma^{\infty}(E\otimes T^*M)$ such that $D_x:=c_x\circ \nabla_x$. By standard elliptic analysis on compact manifolds, we know that $D_x$ is self-adjoint on the first Sobolev space $\text{Dom}(D_x)=W^1(E)$, for each $x\in X$. Lichnerowicz' formula reads $D_x^2=\nabla_x^*\nabla_x+\kappa_x$ for each $x\in X$, where $\Delta_x:=\nabla_x^*\nabla_x$ is a Laplacian-type operator built from the connection $\nabla_x$, and $\kappa_x\in \text{End}(E)$ is an endomorphism built out of the curvature. Suppose there exists a compact subset $K\subset X$ and $c>0$ such that $\kappa_x\geq c>0$ fiberwise, for every $x\in X\backslash K$. Since $\Delta_x\geq 0$ on $L^2(E),$ for every $x\in X$, it follows that $D_x^2\geq c >0$ on $L^2(E)$, for every $x\in X\backslash K$. It then follows that Assumptions \ref{assump1}-\ref{assump3} are satisfied (with $G=\{e\}$ the trivial group, so then $\ee^0(E)=L^2(E)$), so that the $K\!K$-theoretic $\{e\}$-spectral flow class 
	\begin{equation*}
	\SF_{\{e\}}(\DD)\in K\!K^1(\com,C_0(X))
	\end{equation*}
	 is well-defined. In the case $X=\re$, one recovers classical spectral flow as an integer (cf. \cite[Proposition 2.21]{koen}):
	\begin{equation*}
		\Sf_{\{e\}}(\DD)\in K\!K^1(\com,C_0(\re))\simeq K\!K^0(\com,\com)\simeq \mathbb{Z}.
	\end{equation*}
	In particular, the isomorphism $K\!K^1(\com,C_0(\re))\simeq K\!K^0(\com,\com)$ is implemented via Kasparov product on the right by the $K$-homology class $[-i\partial_t]\in K\!K^1(C_0(\re),\com)$, so that 
	\begin{equation*}
		\SF_{\{e\}}(\DD)\otimes_{C_0(\re)}[-i\partial_t]=\Sf_{\{e\}}(\DD).
	\end{equation*}
		\end{example}
	\begin{example}
	Let $M$ be a compact spin manifold and let $\{g_x\}_{x\in X}$ be a smooth family of Riemannian metrics on $M$. Then the corresponding family of spin Dirac operators $\{D_{x}\}_{x\in X}$ is norm continuous (cf. \cite[Proposition 4.22]{van_den_Dungen_2016} and Remark \ref{diffimpliescont}). Suppose that there is a compact subset $K\subset X$ such that for all $x\in X\backslash K$, the scalar curvature $\kappa_x$ associated to $g_x$ is strictly positive ($\kappa_x\geq c>0$). Example \ref{excpt} shows then that there exists a well-defined $K\!K$-theoretic $\{e\}$-spectral flow class in $K\!K^1(\com,C_0(X))$. We use this construction later on to relate $G$-spectral flow to higher rho and eta invariants (Theorem \ref{theoremmain}). 
\end{example}

\begin{example}\label{univcover}
	Let $\widehat{M}$ be a compact Riemannian manifold, $\widehat{E}\to \widehat{M}$ a Hermitian vector bundle and $\{\widehat{D}_x\}_{x\in X}$ a family of symmetric, elliptic, first-order differential operators on $\widehat{E}$. We suppose that Assumptions \ref{assump2} and \ref{assump3} are satisfied by the family $\{\widehat{D}_x\}_{x\in X}$. It thus defines a $K\!K$-theoretic $\{e\}$-spectral flow class
	\begin{equation}\label{eqclassspecflo}
		\text{SF}_{\{e\}}(\widehat{D}_\bullet)\in K\!K^1(\com,C_0(X)).
	\end{equation}
	Let $M$ denote the universal cover of $\widehat{M}$ and $\Gamma:=\pi_1(\widehat{M})$ its fundamental group. Lift $\widehat{E}\to \widehat{M}$ to a Hermitian vector bundle $E\to M$, and the family of operators on $\widehat{M}$ to a family $\{D_x\colon \cs \to L^2(E)\}_{x\in X}$ of $\Gamma$-invariant, symmetric, elliptic, first-order differential operators on $M$. We prove that Assumptions \ref{assump1}-\ref{assump3} are also satisfied by the family $\{D_x\}_{x\in X}$ in Lemma \ref{lemma lift assump univ cov}. Since by construction $M/\Gamma=\widehat{M}$ is compact, the class
		\begin{equation}\label{eqequivspflo}
		\text{SF}_{\Gamma}(\DD)\in K\!K^1(\com,C_0(X,\CS(\Gamma))).
	\end{equation}
	is well-defined. We prove in Subsection \ref{subsecintgspecflow} below that the equivariant spectral flow class (\ref{eqequivspflo}) refines classical spectral flow (\ref{eqclassspecflo}) under a summation map (Corollary \ref{corollary summation map}).

\end{example}

\subsection{$G$-Spectral Flow as a Refinement}\label{subsec gsf as ref}

Let $\co\colon M\to [0,1]$ be a cutoff function, as in (\ref{eq cutoff c}). Consider the space $\gi$ of $G$-invariant smooth sections on $E$. Since $M/G$ is compact, we may choose $\co$ to be compactly supported. We can then make sense of the inner product
\begin{equation}\label{eq l2 t inner prod}
	\langle \co s_1,\co s_2\rangle_{L^2(E)},
\end{equation}
for each $s_1,s_2\in \gi$. The vector space completion of $\co\gi$ with respect to the norm defined by means of (\ref{eq l2 t inner prod}) is denoted by $L^2_T(E)^G$. For each $x\in X$, let  $\widetilde{D}_x\colon \co\gi\to L^2_{T}(E)^G$ denote the map given by
\begin{equation*}
    \widetilde{D}_x(\co s)=\co D_xs,
\end{equation*}
for each $s\in \gi$. We note here that the operator $\widetilde{D}_x$ is symmetric (see Remark \ref{remark d tilde is symmetric}). Denote the completion of $\co\gi$ with respect to the graph norm defined by $\widetilde{D}_x$ by $W^1_{T,x}(E)^G$. Since $G$ is unimodular, the spaces $W^1_{T,x}(E)^G$ and $L^2_T(E)^G$ do not depend on the choice of cutoff function (\cite[Lemma 4.4]{petermathai}).  

Let $\psi \in C_0(X)$ be strictly positive, and constant $1$ on the compact set $K$ of Assumption (3) in Subsection \ref{sec SFG}. Consider the $C_0(X)$-module $C_0(X,L^2_T(E)^G)$ and its subspace
\begeq{eq dom D tilde dot}
\{\eta\in C_0(X,L^2_T(E)^G):\eta(x)\in W^1_{T,x}(E)^G, \;x \mapsto \widetilde{D}_x \eta(x) \in \psi C_0(X,L^2_T(E)^G)\}.
\eneq
Define the unbounded operator $\psi^{-1}\widetilde{D}_{\bullet}$ on $C_0(X,\ee^0(E))$ by
\[
	(\psi^{-1}\widetilde{D}_\bullet\eta)(x):=\psi^{-1}(x)\cdot D_x(\eta(x)),
\]
 for all $\eta$ in \eqref{eq dom D tilde dot}.

Let $\phi\colon C_c(G)\to \com$ denote the integration map on compactly supported functions on $G$. It extends to a $*$-homomorphism
\begin{equation*}
	\phi\colon \CS_{\text{max}}(G)\to\com,
\end{equation*}
which in turn induces a group homomorphism 
\begin{equation}\label{phiind}
\phi_*\colon K_0(\CS_{\text{max}}(G))\to K_0(\com).
\end{equation}
In $K\!K$-theoretic language, the $*$-homomorphism $\phi\colon \CS_{\text{max}}(G)\to \com$ defines a class 
\begin{equation*}
	[\phi_*]:=[(\com,0,\phi)]\in K\!K^0(\CS_{\text{max}}(G),\com),
\end{equation*}
represented by the unbounded Kasparov $(\CS_{\rm{max}}(G),\com)$-cycle $(\CS_{\rm{max}}(G),0,\phi)$. The homomorphism (\ref{phiind}) is given by the Kasparov product on the right with $[\phi_*]$, via the identifications $K_0(\CS_{\text{max}}(G))\simeq K\!K^0(\com,\CS_{\text{max}}(G))$ and $K_0(\com)\simeq K\!K^0(\com,\com)$. 

\begin{proposition}\label{intgspecflo}
The triple $(C_0(X,L^2_T(E)^G), \psi^{-1} \widetilde{D}_\bullet, 1_{\com})$ is an unbounded Kasparov $(\com, C_0(X))$-cycle which represents the class $\SF_G(D_\bullet)\otimes_{\CS_{\rm{max}}(G)}[\phi_*]\in K\!K^1(\com,C_0(X))$.
\end{proposition}
This proposition is proved in Subsection \ref{subsecintgspecflow}.

\begin{remark}
	Proposition \ref{intgspecflo}  is reminiscent of  \cite[Theorem 1.1]{MV03}, \cite[Theorem A.1]{HL08},
 \cite[Proposition D.3]{Mathai_2010} and \cite[Theorem 3.9]{Guo_2021}.
\end{remark}

Recall the construction in Example \ref{univcover}, where the family $\{\widehat{D}_x\colon \Gamma_c^{\infty}(\widehat{E})\to L^2(\widehat{E})\}_{x\in X}$ of symmetric, elliptic, first-order differential operators on a compact manifold $\widehat{M}$ is lifted to a family $\{D_x\colon \cs\to L^2(E)\}_{x\in X}$ of $\Gamma$-invariant, symmetric, elliptic, first-order differential operators on the universal cover $M$, where $\Gamma:=\pi_1(\widehat{M})$ is given by the fundamental group of $\widehat{M}$. 
\begin{corollary}\label{corollary summation map}
	The family $\DF$ satisfies Assumptions \ref{assump1}-\ref{assump3} and 
	\begin{equation*}
		\SF_\Gamma(\DD)\otimes_{\CS_{\rm{max}}(\Gamma)}[\phi_*]=\SF_{\{e\}}(\widehat{D}_\bullet)\in K\!K^1(\com,C_0(X)).
	\end{equation*}
	In particular, if $X=\re$, then $\Gamma$-spectral flow refines classical spectral flow under the summation map $\phi\colon C_c(\Gamma)\to \com$: 
	\begin{equation*}
		\phi_*(\Sf_{\Gamma}(\DD))=\Sf_{\{e\}}(\widehat{D}_\bullet)\in \mathbb{Z}.
	\end{equation*}
	\end{corollary}
Corollary \ref{corollary summation map} is proved at the end of Subsection \ref{subsecintgspecflow}.

\subsection{$G$-Spectral Flow as an Index}\label{sec SF index}

In this subsection we assume that $X$ is a connected, complete Riemannian manifold. We consider a symmetric, elliptic, first-order differential operator $P_X$ on a Hermitian vector bundle $F\to X$ with finite propagation speed, and recall that $P_X$ defines a $K$-homology class $[P_X]\in K\!K^1(C_0(X),\com)$. 		

Consider the tensor product 
\begin{equation*}
\ee^0(E)\otimes L^2(X,F),
\end{equation*} 
and note that the operator $\DD\otimes 1$ is well-defined on the domain $\text{Dom}(\DD)\otimes_{C_0(X)}L^2(X,F)$ in $$C_0(X,\ee^0(E))\otimes_{C_0(X)} L^2(X,F) \simeq \ee^0(E)\otimes L^2(X,F).$$ By  \cite[Lemma 9.10]{lance}, $\DD\otimes 1$ is regular self-adjoint on $\ee^0(E)\otimes L^2(X,F)$. We denote the operator $\DD\otimes 1$  simply by $\DD$ from now on. The operator $1\otimes P_X$ may be directly defined on the domain $\ee^0(E)\otimes \text{Dom}(P_X)$.  We denote such operator simply by $P_X$, which by \cite[Theorem 5.4]{kl1} is regular self-adjoint on $\ee^0(E)\otimes L^2(X,F)$. The (generalised) \textit{Dirac--Schrödinger operator} is then defined as
\begin{equation}\label{eq def PD}
    P_{D}:=P_X-i\DD,
\end{equation}
on the domain $\text{Dom}(P_X)\cap \text{Dom}(\DD)\subset \ee^0(E)\otimes L^2(X,F)$. 
\begin{remark}\label{remark 2}
	Under the isomorphism
	\begin{equation}\label{eq identification hilb mod}
\ee^0(E)\otimes L^2(X,F)\simeq C_0(X,\ee^0(E))\otimes_{C_0(X)} L^2(X,F),
\end{equation}
the operator $1\otimes P_X$ can be alternatively defined as follows: one may extend the exterior derivative on the space $C_0^1(X)$ of differentiable functions vanishing at infinity with derivatives vanishing at infinity to an operator
\begin{equation*}
	d\colon C_0^1(X,\ee^0(E))\simeq C_0^1(X)\otimes \ee^0(E) \xrightarrow{d\otimes 1} \Gamma_0(T^*X)\otimes \ee^0(E)\simeq \Gamma_0(T^*X\otimes \ee^0(E)),
\end{equation*}
where $\Gamma_0(T^*X)$ denotes the space of continuous sections on the cotangent bundle $T^*X$ which vanish at infinity. Let $\sigma_{P_X}:T^*X\to \End(F)$ denote the principal symbol of $P_X$. Following \cite[\S3.2]{koen}, we see that the operator $1\otimes_d P_X$ on the initial domain $C_c^1(X,\ee^0(E))\otimes_{C^1_0(X)}\Dom(P_X)$, given by
\begin{equation}\label{eq px}
	(1\otimes_d P_X)(\eta\otimes \xi):= \eta\otimes P_X\xi+ (\sigma_{P_X}\otimes 1)(d\eta)(\xi),
\end{equation}
agrees with $1\otimes P_X$ on $\ee^0(E)\otimes L^2(X,F)$ under the identification (\ref{eq identification hilb mod}). The operator $1\otimes_d P_X$ is used in \cite{koen} for some explicit calculations, but for us it shall be enough to consider the operator $1\otimes P_X$ on $\ee^0(E)\otimes \Dom(P_X)$.
\end{remark}

In what follows, we study the index theory of the Dirac--Schrödinger operator \eqref{eq def PD}  and its relation to $G$-spectral flow through two different approaches. In Subsection \ref{subseccontfam} we follow \cite{koen} and consider the family $\DF$ satisfying Assumptions \ref{assump1}-\ref{assump3} (the continuous family case), with the extra assumption that the family is assumed to be \textit{locally constant outside of the compact set} $K\subset X$ (see Assumption \ref{assump4}). The second case is treated in Subsection \ref{subsecdiffam}, where we follow the methods in \cite{kl1,koen} and assume that the family is suitably differentiable. The former approach will be enough for us to go forward with the study of the (equivariant) index theory of the Dirac--Schrödinger operator in Section \ref{seclocind}, and follows directly from the results in \cite{koen}, whereas the latter is slightly more general, but requires a bit more work to carry out.

In order to prove an ``index equals spectral flow" theorem in \cite[\S5]{koen}, an extra assumption is considered, namely (A4'). We consider the following stronger assumption:
\begin{enumerate}[(4)]
	\item\label{assump4} The family $\DF$ is \textit{locally constant outside of} $K\subset X$:  there exists a disjoint finite open cover $\{U_j\}$ of $X\backslash K$ such that $D_x=D_y$, whenever $x,y\in U_j$ are in the same open set.
\end{enumerate}

The following result is now an immediate consequence of \cite[Theorem 5.15]{koen}: 
\begin{theorem}\label{propfredindeqsfg1}
	Suppose the family $\DF$ satisfies Assumptions \ref{assump1}-\ref{assump4}. Then the
operator $P_X-i\DD$ is regular and Fredholm, and 
\begin{equation*}
	\SF_G(\DD)\otimes_{C_0(X)}[P_X]=\Index(P_X-i\DD)\in K_0(\CS(G)).
\end{equation*}
In particular, if $X=\re$ then
\begin{equation*}
	\Sf_G(\DD)=\Index(-i\partial_t-i\DD) \in K_0(\CS(G)).
\end{equation*}
\end{theorem}

We also describe $G$-spectral flow as an index without the assumption that  $\DF$ is locally constant outside a compact set. 
We follow \cite{kl1,koen} and consider differentiable families:
\begin{definition}
 Let $V$ be a normed vector space and $X$ a Riemannian manifold. We say that a map $f\colon X\to V$ is \textit{differentiable} if, for every $x\in X$, there exists a linear map $df_x\colon T_xX\to V$ such that
\begin{equation*}
	\lim_{h\to 0} \frac{\left\|f(x)-f(x+h)-df_x(h)\right \|_{V}}{\|h\|_{T_xX}}=0,
\end{equation*}
for every $h \in T_xX$, where $x+h:=\exp_x(h)$ is given by the exponential map around $x\in X$.
\end{definition}
\begin{definition}\label{defweakdif}
	Let $X$ be a smooth manifold and let $\ee_1, \ee_2$ be countably generated Hilbert $B$-modules. A map $S\colon X\to \mathcal{L}(\ee_1,\ee_2)\colon x\mapsto S_x$ is said to have \textit{uniformly bounded weak derivative} if, for each $e_1\in \ee_1$ and $e_2\in \ee_2$
	\begin{enumerate}[1.]
		\item The map $f\colon X\to B\colon x\mapsto \langle S_x e_1,e_2\rangle_{\ee_2}$ is differentiable;
		\item For every $x\in X$, the weak derivative $dS(x)\colon\ee_1\to\ee_2\otimes T_x^*X$, defined for each $h\in T_x X$ by the relation 
	\begin{equation*}
	\langle dS(x)(h)e_1,e_2\rangle_{\ee_2}=df_x(h),\end{equation*}
is a bounded operator; 
		\item The supremum $\sup_{y\in X}\|dS(y)\|_{\mathcal{L}(\ee_1,\ee_2\otimes T_x^*X)}$ is finite.
	\end{enumerate}
\end{definition}
\begin{remark}\label{diffimpliescont}
	If $D\colon X\to \mathcal{B}(W^1(E),L^2(E))$ has uniformly bounded weak derivative, then Assumption \ref{assump2} in Subsection \ref{sec SFG} is automatically satisfied. Indeed, following \cite[Remark 8.4.2]{kl1}, take $\eta_1\in W^1(E)$, $\eta_2\in L^2(E)$ and denote by $\gamma(x,y)$ the geodesic curve between two points $x,y\in X$ in the same geodesic coordinate chart. It follows that
	\begin{equation}\label{contfromdiff}
		\begin{split}
			|\langle (D_x-D_y)\eta_1,\eta_2\rangle_{L^2(E)}|&=\left |\int_{\gamma(x,y)}\langle dD(s) \eta_1,\eta_2\rangle_{L^2(E)}\,ds\right|\\
			&\leq \displaystyle\sup_{y\in X} \|dD(y)\|\cdot \|\eta_1\|_{W^1(E)}\cdot\|\eta_2\|_{L^2(E)}\cdot\text{dist}(x,y),
		\end{split}
	\end{equation}
	so that  $\|D_x-D_y\|_{\mathcal{B}(W^1(E),L^2(E))}\leq C\cdot \text{dist}(x,y)$.
\end{remark}

We consider the modified family $\{D_x':=\psi^{-1}(x)D_x\}_{x\in X}$, where $\psi\in C_0(X)$ is a strictly positive function with $\psi(x)=1$ for all $x\in K$. Analogously to  \eqref{eq def PD}, 
 we construct the Dirac--Schrödinger operator $P_X - i \DD'$ with domain $(\Dom(P_X)\cap \Dom(\DD'))\subset \ee^0(E)\otimes L^2(X,F)$. 
 To ensure that the domains of the relevant operators are well-behaved, we work with the assumption that the family $\{D_x\}_{x\in X}$ has uniformly equivalent graph norms.

 \begin{definition}\label{def unif graph norm}
Let $\ee$ be a Hilbert module over a $C^*$-algebra $B$. Let $A$ be a set and, for each $a \in A$, let $D_a$ be an unbounded operator on $\ee$ with domain $\mathcal{D}$ independent of $a$. 
 Then, the family of operators $(D_a)_{a \in A}$ has \emph{uniformly equivalent graph norms} if there are $a_0 \in A$ and  $C_1, C_2>0$ such that for all $a \in A$ and all $v \in \mathcal{D}$,
 \[
 C_1\|(D_{a_0}+i) v\|_{\ee} \leq  \|(D_{a}+i) v\|_{\ee} \leq  C_2 \|(D_{a_0}+i) v\|_{\ee}. 
 \] 
 \end{definition}
 
\begin{theorem}\label{propindtheorem}
Suppose that the family $\{D_x\colon \cs\to L^2(E)\}_{x\in X}$ has uniformly equivalent graph norms
 and uniformly bounded weak derivative. Then $P_X-i\DD'$ is regular and Fredholm, and
\begin{equation*}
	\SF_G(\DD)\otimes_{C_0(X)}[P_X]= \Index(P_X-i\DD') \in K_0(\CS(G)).
\end{equation*}
In particular, if $X=\re$ then
\begin{equation*}
	\Sf_G(\DD)=\Index(-i\partial_t-i\DD')  \in K_0(\CS(G)).
\end{equation*}
\end{theorem}

\begin{remark}\label{remarkindeqsf}
By \cite[Lemma 5.8]{koen}, in the setting of Theorem \ref{propindtheorem} we see that there exists $\lambda_0\geq 1$ such that, for all $\lambda\geq \lambda_0$, the operator $P_X-i\lambda\DD$ is regular and Fredholm on $\ee^0(E)\otimes L^2(X,F)$. By \cite[Proposition 5.10]{koen} and Theorem \ref{propindtheorem}, it follows that 
\begin{equation}\label{eqfredind}
	\SF_G(\DD)\otimes_{C_0(X)}[P_X]=\Index(P_X-i\lambda\DD) \in K_0(\CS(G)).
\end{equation}
In particular, if $X=\re$ then
\begin{equation*}
	\Sf_G(\DD)  =\Index(-i\partial_t-i\lambda\DD).
\end{equation*}
\end{remark}

Theorem \ref{propindtheorem} is proved at the end of Subsection \ref{subsecdiffam}.

\subsection{$G$-Spectral Flow and Secondary Invariants} \label{sec rho eta invar}

For our final results we show how equivariant spectral flow relates  delocalised $\eta$-invariants and $\rho$-invariants associated to (spin Dirac operators for) different positive scalar curvature metrics to each other.

We consider the following setting. Let $M$ be an odd-dimensional Riemannian manifold and $G$ a unimodular, locally compact group acting properly, isometrically and cocompactly on $M$. We assume here that $M$ admits a $G$-spin structure. Let $\{g_t\}_{t\in\re}$ be a smooth family of $G$-invariant Riemannian metrics on $M$ such that $g_0$ and $g_1$ have positive scalar curvature, and such that 
\begin{equation*}
	\begin{cases}
	g_t=g_0,\;\text{for all}\; t\leq 0,\\
	g_t=g_1,\;\text{for all}\; t\geq 1.
\end{cases}
\end{equation*}
 For each Riemannian manifold $(M,g_t)$, let $S_t$ denote the spinor bundle over $M$ corresponding to $g_t$, let $c_t\colon T^*M\to \End(S_t)$ denote the Clifford action and $\nabla^{S_t}$ the $G$-equivariant spin connection on $S_t$. We define the spin Dirac operators 
\begin{equation}\label{eq spin dirac op mt}
	\slashed{D}_{M_t} := c_t\circ \nabla^{S_t},
\end{equation}
which are $G$-equivariant, symmetric, elliptic, first-order differential operators on $S_t$. By Lichnerowicz' formula, there holds
\begin{equation}\label{eq weitzenbock}
	\slashed{D}_{M_t}=(\nabla^{S_t})^*\nabla^{S_t}+\slashed{\kappa}_t,
\end{equation}
 where $\slashed{\kappa}_t$ is given by a positive multiple of the scalar curvature of $g_t$. 
 
 Consider now the product manifold $M\times \re$ and endow it with the metric $dt^2+g_t$. We set $M_t:= M\times\{t\}$ and, following \cite[\S5]{bgm}, we see that there exists a spinor bundle $\widetilde{S}$ over $M\times \re$ such that
\begin{equation}\label{eqrestriction}
	\widetilde{S}\big{|}_{M_t}=S_t\oplus S_t,
\end{equation}
as Hermitian vector bundles. Let $\nabla^{\widetilde{S}}$ denote the spin connection on $\widetilde{S}$, which preserves the splitting (\ref{eqrestriction}). For each $m\in M$, the parallel transport over the curve $t\mapsto (m,t)$ is a $G$-equivariant isometry 
\begin{equation}\label{eqisometry}
	\tau^0_t\colon (S_0)_m\to (S_t)_m,
\end{equation}
which extends to a $G$-equivariant isometry $\tau_t\colon  S_0\to S_t$ over $M$, for each $t\in \re$. We set 
\begin{equation}\label{eqdiracops}
D_t:=\tau_t^{-1}\circ \slashed{D}_{M_t}\circ \tau_t	,
\end{equation}
so that we get a family 
\begeq{eq Dt Spin}
\{D_t\colon \Gamma_c^{\infty}(M,S_0)\to L^2(M,S_0)\}_{t\in \re} 
\eneq
of $G$-equivariant, symmetric, elliptic, first-order differential operators on the spinor bundle $S_0$ over $M$. Here we suppose, additionally, that the graph norms of the operators $D_t\colon \Gamma_c^{\infty}(M,S_0)\to L^2(M,S_0)$ are equivalent (as they are when $G$ is discrete and acts freely; see Proposition \ref{prop garding g discrete}).

\begin{proposition}\label{prop Spin SFG}
The family \ref{eq Dt Spin} satisfies the conditions of Theorem \ref{propfredindeqsfg1}.
\end{proposition}
This proposition is proved in Subsection \ref{subsecspindiracop}. In particular, the equivariant spectral flow 
\[
\Sf_G(\DD) \in K_0(C^*_r(G))
\]
is defined.

 Let $h \in G$ and suppose that its centraliser $Z$ is unimodular. Suppose there is a dense Fr\'echet subalgebra $\A(G) \subset C^*(G)$ of functions on $G$ such that for all $f \in \A(G)$, the \emph{orbital integral}
 \[
 \tau_h(f) = \int_{G/Z} f(xhx^{-1})\, d(xZ)
 \]
 converges. Suppose, furthermore, that this defines a continuous linear functional on $\A(G)$, and that $\A(G)$ is closed under holomorphic functional calculus. Then $\tau_h$ is a trace on $\A(G)$ (see \cite[Lemma 2.2]{HW18}), and we obtain an induced map
 \begeq{eq tauh K}
 \tau_h\colon K_0(C^*(G)) = K_0(\A(G)) \to \mathbb{C}.
 \eneq

 \begin{theorem}\label{thm eta}
Suppose that either
\begin{itemize}
\item $G$ is discrete and finitely generated, the conjugacy class of $h$ has polynomial growth, and that either the action on $M$ is free, or $G$ has slow enough exponential growth;
\item $G$ is a connected, real semisimple Lie group, and $h$ is semisimple; or
\item $h=e$.
\end{itemize}
Then an algebra $\A(G)$ as in \eqref{eq tauh K} exists, 
the integrals (\ref{eq def eta g}) defining $\eta_h(D_0)$ and  $\eta_h(D_1)$ converge, and 
\[
\tau_h
(\Sf_{G}(\DD) )=\int_{M^h} f_h \frac{p^*\hat A(M^h)}{\det(1-hR^{N})^{1/2}}+\frac{1}{2}(\eta_h(D_0) - \eta_h(D_1)).
\]
\end{theorem}
\begin{remark}
For the condition in the first case of Theorem \ref{thm eta} that a finitely generated, discrete group $G$ have slow enough exponential growth, see the second case of Corollary 2.10 in \cite{HWW22}.
\end{remark}

For a relation with $\rho$-invariants, we consider
 a smooth family $\{\hat{g}_t\}_{t\in \re}$  of Riemannian metrics on an odd dimensional compact spin manifold $\widehat{M}$. We assume that $\hat{g}_t=\hat{g}_0$ if $t\leq 0$ and $\hat{g}_t=\hat{g}_1$ if $t\geq 1$. We consider the universal cover of $\widehat{M}$, which we denote by $M$, and set $\Gamma:=\pi_1(\widehat{M})$, which acts properly, freely, isometrically and cocompactly on $M$. We then lift the family of metrics $\{\hat{g}_t\}_{t\in \re}$ to $\Gamma$-invariant metrics $\{g_t\}_{t\in \re }$ on $M$, and make the same construction as in the beginning of Subsection \ref{sec rho eta invar}, yielding a family of operators $\{D_t\}_{t\in \re}$ and $\Gamma$-spectral flow
 \begin{equation*}
 	\Sf_\Gamma(\DD)\in K_0(\CS_r(\Gamma)).
 \end{equation*}
 
 Consider the inclusion maps $j_1, j_2\colon M \hookrightarrow M \times [0,1]$, given by $j_k(m) = (m,k)$ for all $k \in \{0,1\}$ and $m \in M$. They induce maps
 \[
 \begin{split}
 (j_1)_*, (j_2)_*\colon K_0(D^*(M)^{\Gamma}) \to K_0(D^*(M \times [0,1])^{\Gamma})
 \end{split}
 \]
 as in Definition 1.6 in \cite{piazzaschick} and the text below it. Furthermore, the inclusion map $\iota\colon C^*(M\times [0,1])^{\Gamma} \hookrightarrow D^*(M\times [0,1])^{\Gamma}$ induces
 \[
 \iota_*\colon K_0(C^*_r(\Gamma)) = K_0(C^*(M\times [0,1])^{\Gamma}) \to K_0( D^*(M\times [0,1])^{\Gamma}).
 \] 
 For the equality of $K$-theory groups, which relies on compactness of $M/\Gamma$, see 
 Theorem 5.3.2 in \cite{WillettYu}. 

 \begin{theorem}\label{theoremmain}
 We have
\[
 (j_1)_*(\rho(g_1)) - (j_0)_*(\rho(g_0)) = 
\iota_*(\Sf_{\Gamma}(\DD)).
\]
\end{theorem}

Theorems \ref{thm eta} and \ref{theoremmain} are proved in Subsection \ref{subsecdelocrhoeta}. We point out some consequences.
 \begin{corollary}\label{corollary 1}
In the setting of Theorem \ref{theoremmain}, the class $\iota_*(\Sf_\Gamma(D_\bullet))$ does not depend on the path $(g_t)_{t \in [0,1]}$, only on its endpoints. In the setting of Theorem \ref{thm eta}, suppose that $M^h=\emptyset$. Then, the number $\tau_{h*}(\Sf_G(D_\bullet))$ does not depend on the path $\{D_t\}_{t\in [0,1]}$, only on its endpoints.
 \end{corollary}
 
 Theorems \ref{thm eta} and \ref{theoremmain} imply the following homotopy invariance properties of delocalised $\eta$-invariants and $\rho$-invariants. These also follow directly from the relevant higher APS-index theorems, which are used in the proofs of  Theorems \ref{thm eta} and \ref{theoremmain}. (See Corollary 1.16 in \cite{piazzaschick} for a stronger result in the case of $\rho$-invariants). This is a re-interpretation in terms of a natural vanishing property of $G$-spectral flow, not an independent proof. 
 \begin{corollary}\label{corollary 2}
Suppose that $g_0$ and $g_1$ can be connected by a smooth path of $G$-invariant Riemannian metrics of positive scalar curvature. In the situation of  Theorem \ref{thm eta}, assume moreover that $G$ acts freely on $M$ and $h\neq e$. Then, we have $\eta_h(D_0) = \eta_h(D_1)$. In the situation of  Theorem \ref{theoremmain}, we have $(j_0)_*(\rho(g_0))  =  (j_1)_*(\rho(g_1))$.
 \end{corollary}
\begin{proof}
	The result follows immediately from Proposition \ref{prop k equals 0} and Theorems \ref{thm eta} and \ref{theoremmain}.
\end{proof}

\section{Construction of $G$-Spectral Flow}\label{secgspectralflow}
 
 In this section, we prove Propositions \ref{prop D dot Fred} and \ref{intgspecflo}, showing that equivariant spectral flow is a well-defined refinement of classical spectral flow.
 
  Let $M$, $E$, $G$ and $D$ be as at the start of Section \ref{secpreliminaries}. 
   In Subsections \ref{subsecadjointpos} and \ref{subseceqellana} we do not yet make the assumptions made in Subsection \ref{sec SFG}, such as compactness of $M/G$, as we will later apply the material in those subsections to the manifold $M\times X$, on which $G$ does not act cocompactly. The assumptions in Subsection \ref{sec SFG} will be made in Subsections \ref{sec pf SFG} and \ref{subsecintgspecflow}.

\subsection{Adjointability and Positivity}\label{subsecadjointpos}

In this subsection, we follow \cite[\S3]{haoguo} in order to investigate the conditions under which an operator $A\in \mathcal{B}(W^i(E),W^j(E))$ defines an adjointable operator in $\mathcal{L}(\ee^i(E),\ee^j(E))$. We also relate fiberwise positivity of $G$-equivariant endomorphisms with positivity on the module $\ee^0(E)$. 

 We shall need the following lemma:
\begin{lemma}[\rm{\cite[Lemma 3.7]{haoguo}}]\label{haoguo37}
 Suppose that $M/G$ is compact. Let $A\in \mathcal{B}(W^i(E))$ be a $L^2$-positive operator  with compactly supported distributional kernel. Then the element 
	\begin{equation*}\left\langle \left(\int_G g(A)\,dg\right)(s),s\right\rangle_{\ee^i(E)}\in \CS(G)_+
	\end{equation*}	
 is positive in $\CS(G)$ for every $s\in \cs$.
\end{lemma}
Proposition \ref{haoguo35} below is a slighted modified version of \cite[Proposition 3.5]{haoguo}, which still holds true if the locality assumption is replaced by the weaker condition of $A$ having a properly supported Schwartz kernel. In the latter case, the constant $C_{\co}$ in (\ref{l2e0norms}) depends on the support of $\co^2 A^*A+A^*A\co^2$, instead of only on the choice of cutoff function $\co$.
\begin{proposition}\label{haoguo35} 
Suppose that $M/G$ is compact. Let $A\colon \cs\to\cs$ be a $G$-equivariant, local operator which defines a bounded operator in $\mathcal{B}(W^i(E), W^j(E))$. Then, $A$ defines an adjointable operator in $\mathcal{L}(\ee^i(E),\ee^j(E))$, and

	\begin{equation}\label{l2e0norms}
		\|A\|_{\mathcal{L}(\ee^i(E),\ee^j(E))}\leq C_{\co}\cdot\|A\|_{\mathcal{B}(W^i(E),W^j(E))},
	\end{equation}
	where the constant $C_{\co}$ depends only on the choice of cutoff function $\co$.
\end{proposition}
\begin{proof}
	 Let $A^*\in\mathcal{B}(W^j(E),W^i(E))$ denote the adjoint of $A$. The operator $A_1:=(\co^2 A^*A+A^*A\co^2)/2$ is self-adjoint, has compactly supported distributional kernel and is bounded on $W^i(E)$ by $\|A\|^2\cdot\|\co^2\|_{\infty}$, where $\|A\|$ denotes the operator norm of $A\colon W^i(E)\to W^j(E)$. Let $\co_1$ be a non-negative, compactly supported function on $M$ which is identically equal to $1$ on the support of $\co$. Then the operator $A_2:=\co_1^2\|A\|^2\|f^2\|_{\infty}-A_1$ is positive, bounded and has compactly supported distributional kernel (this is where we use the locality assumption). The proof then follows just like in \cite{haoguo}: one applies Lemma \ref{haoguo37} to the operator $A_2$ so that, for every $s\in\cs$,
	\begin{equation*}
		\left\langle s,\left(\int_G g(A_2)\,dg\right)s\right\rangle_{\ee^i(E)}=\left (\int_G g(\co_1^2)\|A\|^2\|\co^2\|_{\infty}\,dg\right)\langle s,s\rangle_{\ee^i(E)}-\langle s,A^*A(s)\rangle_{\ee^i(E)},
	\end{equation*} 
is positive in $\CS(G)$, where we use the fact that $\int_G g(A_1)\,dg=A^*A$. It follows that
\begin{equation*}
\langle A(s),A(s)\rangle_{\ee^j(E)}=\langle s,A^*A(s)\rangle_{\ee^i(E)}\leq C\cdot\|A\|^2\cdot\|\co^2\|_{\infty}\cdot\langle s,s\rangle_{\ee^i(E)},
\end{equation*} 
in $\CS(G)$, where $C:= \int_G g(\co_1^2)\,dg$. Consequently, $A$ extends to an operator in $\mathcal{L}(\ee^i(E),\ee^j(E))$, since by a similar argument one can check that $A^*\colon W^j(E)\to W^i(E)$ defines a bounded adjoint for $A$ of $\mathcal{L}(\ee^j(E),\ee^i(E))$.
\end{proof}

As a corollary, if $M/G$ is compact, then the operator $D^k$ defines an element of $\mathcal{L}(\ee^{k}(E),\ee^0(E))$ for every $k\geq 0$. This result can be extended to non-cocompact actions:\begin{proposition}[\rm{\cite[Proposition 3.8]{haoguo}}]\label{adjointdirac}
 For all $k\geq 0$, the operator $D^k$ defines an element of $\mathcal{L}(\ee^{j+k}(E),\ee^j(E))$. 
\end{proposition}
 We finish this subsection with a proof of the following fact about positivity of $G$-equivariant endomorphisms on $\ee^0(E)$:
\begin{lemma}\label{positivitynoncocomp}
Let $A\in \End(E)^G$ be a $G$-equivariant endomorphism and suppose that $A\geq c>0$ fiberwise. Then $A\geq c>0$ on $\ee^0(E)$, that is,
\begin{equation}
	\langle (A-c)s,s\rangle_{\ee^0(E)},
\end{equation}
for every $s\in \cs$.
\end{lemma}
\begin{proof}
Let $B\in End(E)^G$ be the fiberwise positive square-root of $A-c\in \End(E)^G$. It follows that
\begin{equation*}
	\begin{split}
		\langle (A-c)s,s\rangle_{\ee^0(E)}(g)&=\int_M\langle ((A-c)(m))s(m),(g(s))(m)\rangle_{E_m}\,d\mu(m)\\
		&=\int_M\langle B(m)s(m),(g(Bs))(m)\rangle_{E_m}\, d\mu(m)\\
		&=\langle Bs,Bs\rangle_{\ee^0(E)}(g),
	\end{split}
\end{equation*}
for each $s\in\cs$ and $g\in G$, where we used $G$-equivariance of $B$. The result now follows from noting that $\langle Bs,Bs\rangle_{\ee^0(E)} \geq 0$ on $\CS(G)$.
\end{proof}
\subsection{Equivariant Elliptic Analysis}\label{subseceqellana}

The analogue of the non-compact Rellich Lemma on Sobolev $G$-modules can be stated as follows:

\begin{proposition}[\rm{\cite[Theorem 3.12]{haoguo}}\label{lemmaeqrellich}]
	Let  $f\colon M\to \com$ be a cocompactly supported $G$-invariant function. Then, multiplication by $f$ is a compact operator $\ee^s(E)\to\ee^t(E)$ if $s>t$.
\end{proposition}
\begin{corollary}\label{compactincl}
Suppose that $M/G$ is compact. Then the inclusion $\ee^s(E)\hookrightarrow\ee^t(E)$ is compact if $s>t$.
\end{corollary}

One of the consequences of the elliptic estimate (or Gårding's inequality) on a compact manifold is that any two elliptic, first-order differential operators define equivalent graph norms as maps $\cs \to L^2(E)$. We prove in Proposition \ref{propeqellest} that, if the action is cocompact, then equivalent graph norms on (first) Sobolev spaces imply equivalent graph norms on the level of $G$-Sobolev modules. In particular, for discrete groups acting freely, this result shall be enough for us to derive that two elliptic, first-order differential operators define equivalent graph norms as maps $\cs\to \ee^0(E)$.

In what follows, we write $\widehat{W^1}(E)$ and $\widehat{\ee^1}(E)$ to denote, respectively, the first Sobolev space and first $G$-Sobolev module defined with respect to a elliptic, first-order operator $\WHD\colon \cs\to L^2(E)$.
\begin{proposition}\label{propeqellest}
	Let  $D,\WHD\colon \cs \to L^2(E)$ be two  $G$-equivariant, symmetric, elliptic, first-order differential operators. Suppose that their graph norms are equivalent and that $M/G$ is compact. Then, there exist numbers $C_1,C_2>0$ such that
	\begin{equation*}
		C_1\|s\|_{\widehat{\ee^1}(E)}\leq \|s\|_{\ee^1(E)}\leq C_2\|s\|_{\widehat{\ee^1}(E)},
	\end{equation*}
	for every $s\in \cs $. In particular, $\ee^1(E)=\widehat{\ee^1}(E)$.
\end{proposition}
\begin{proof}
	 By hypothesis, we see that $W^1(E)=\widehat{W^1}(E)$. It follows that $\WHD$ defines a bounded operator in $\mathcal{B}(W^1(E),L^2(E))$, and so the operator $\WHD+i$ is bounded $W^1(E)\to L^2(E)$. By Proposition \ref{haoguo35}, the operator $\WHD+i$ is in $\mathcal{L}(\ee^1(E),\ee^0(E))$, and there exists a constant $C_{\co}$ such that
	\begin{equation}\label{eqdhat}
		\|\WHD+i\|_{\mathcal{L}(\ee^1(E),\ee^0(E))}\leq C_{\co}\|\WHD+i\|_{\mathcal{B}(W^1(E),L^2(E))}.
	\end{equation}
	For every $s\in \cs $, 
		\begin{equation}\label{equivgraphnorms}
		\begin{split}
			\|s\|_{\widehat{\ee^1}(E)}&=\|(\WHD+i)s\|_{\ee^0(E)}\\
			&\leq \|\WHD+i\|_{\mathcal{L}(\ee^1(E),\ee^0(E))}\|s\|_{\ee^1(E)}\\
			&\leq C_{\co}\|\WHD+i\|_{\mathcal{B}(W^1(E),L^2(E))}\|s\|_{\ee^1(E)},
		\end{split}
	\end{equation}
	which proves the first inequality. The other one follows similarly.
\end{proof}

Suppose now, additionally, that $G=\Gamma$ is a discrete group and that it acts freely on $M$. We can then choose a fundamental domain $U\subset M$ for the $\Gamma$-action on $M$, that is, a relatively compact open set such that $\gamma U\cap U = \emptyset$ for every $\gamma\neq e$, and such that $M\backslash(\cup_{\gamma\in\Gamma} \gamma U)$ has measure zero.
\begin{proposition}\label{prop garding g discrete}
	Suppose that $\Gamma$ is a discrete group acting properly, isometrically, freely and cocompactly on $M$. Let $D,\WHD\colon \cs \to L^2(E)$ be two  $\Gamma$-equivariant, symmetric, elliptic, first-order differential operators. Then the graph norms of $D$ and $\widehat{D}$ are equivalent.
\end{proposition}
\begin{proof}
	For every  $s\in \cs $,
	\begin{equation}\label{eq d plus i}
		\begin{split}
			\|(D+i)s\|^2_{L^2(E)}&=\int_M\|(D+i)s(m)\|_{E_m}\,d\mu(m)\\
			&=\sum_{\gamma\in\Gamma}\int_{\gamma U}\|((D+i)s)|_{\gamma U}(m)\|_{E_m}\,d\mu(m)\\
			&=\sum_{\gamma\in\Gamma}\int_{\gamma U}\|(\gamma^{-1}\cdot\gamma)((D+i)s)|_{\gamma U})(m)\|_{E_m}\,d\mu(m)\\
			&=\sum_{\gamma\in\Gamma}\int_{\gamma U}\|\gamma^{-1}(\gamma((D+i)s|_{\gamma U})(\gamma m))\|_{E_{ m}}\,d\mu(m)\\
			&=\sum_{\gamma\in\Gamma}\int_{\gamma U}\|\gamma^{-1}((D+i)\gamma(s|_{\gamma U})(\gamma m))\|_{E_{ m}}\,d\mu(m),\\
		\end{split}
	\end{equation}
	where we used $\Gamma$-equivariance of $D+i$ for the last equality. By the fact that $\Gamma$ acts via isometries on the fibers of $E$, by $\Gamma$-invariance of $d\mu$, and by writing $m'=\gamma m$, we see that (\ref{eq d plus i}) is equal to
	\begin{equation*}
		\begin{split}
			\|(D+i)s\|^2_{L^2(E)}&=\sum_{\gamma\in\Gamma}\int_{\gamma U}\|(D+i)\gamma(s|_{\gamma U})(m')\|_{E_{m'}}\,d\mu(m')\\
			&=\sum_{\gamma \in \Gamma}\|(D+i)\gamma(s|_{\gamma U})\|_{L^2(E|_{U})}.
		\end{split}
	\end{equation*}
	Here we note that the section $\gamma(s|{\gamma U})$ is supported within $U$, for every $\gamma\in \Gamma$. By using Gårding's inequality within the compact $\overline{U}$ (\cite[10.4.4]{higsonroe}), we see that there is a constant $C>0$ such that
	\begin{equation*}
		\begin{split}
			\|(D+i)s\|^2_{L^2(E)}&=\sum_{\gamma \in \Gamma}\|(D+i)\gamma(s|_{\gamma U})\|_{L^2(E|_{U})}\\
			&\leq C\cdot\sum_{\gamma \in \Gamma}\|(\widehat{D}+i)\gamma(s|_{\gamma U})\|_{L^2(E|_{U})}\\
			&=C\cdot\|(\widehat{D}+i)s\|^2_{L^2(E)},
		\end{split}
	\end{equation*}
	which finishes the proof.
\end{proof}
The following is an immediate consequence of Propositions \ref{propeqellest} and \ref{prop garding g discrete}.
\begin{corollary}\label{gardingine} 
Suppose that $\Gamma$ is a discrete group acting properly, isometrically, freely and cocompactly on $M$. Let $D,\WHD\colon \cs\to L^2(E)$ be two  $\Gamma$-equivariant, symmetric, elliptic, first-order differential operators. Then the $\Gamma$-Sobolev modules $\ee^1(E)$ and $\widehat{\ee^1}(E)$, defined respectively by $D$ and $\WHD$, are the same.
\end{corollary}

\subsection{Well-definedness of $G$-Spectral Flow}\label{sec pf SFG}

We return to the setting of Subsection \ref{sec SFG} and make Assumptions (1)--(3) listed at the start of that subsection, along with cocompactness of the $G$-action on $M$.

\begin{lemma}\label{a1lemma}
The family $\{D_x\}_{x\in X}$ consists of regular self-adjoint operators on $\ee^0(E)$, with domains $\ee^1_x(E)=:\ee^1(E)$ independent of $x\in X$. Moreover, the inclusion $\ee^1(E)\hookrightarrow \ee^0(E)$ is compact.
\end{lemma}
\begin{proof}
By Proposition \ref{adjointdirac} we see that $D_x\in \mathcal{L}(\ee^1_x(E),\ee^0(E))$ for every $x\in X$. Also, by \cite[Theorem 5.8]{kasparov}  each $D_x$ is a regular self-adjoint operator on $\ee^0(E)$. From Proposition \ref{propeqellest} and Assumption \ref{assump1}, we see that all domains $\ee^1_x(E)$ coincide, so from now on we may write $\ee^1_x(E)=:\ee^1(E)$ and equip $\ee^1(E)$ with the graph norm of $D_{x_0}$, for some $x_0\in X$. We thus have a family $\{D_x\colon\ee^1(E)\to \ee^0(E)\}_{x\in X}$ of regular self-adjoint operators all of which are defined on the same domain $\ee^1(E)$. Finally, we see by cocompactness of the action and Corollary \ref{compactincl} that $\ee^1(E) \hookrightarrow \ee^0(E)$ is compact, which finishes the proof.
\end{proof}

\begin{lemma}\label{lem UEGN}
	Suppose that the family $\{D_x\colon W^1(E)\to L^2(E)\}_{x\in X}$ has uniformly equivalent graph norms on $L^2(E)$. Then, the induced family $\{D_x\colon \ee^1(E)\to \ee^0(E)\}_{x\in X}$ has uniformly equivalent graph norms on $\ee^0(E)$.
\end{lemma}
\begin{proof}
	Let $C_1,C_2>0$ be such that, for every $\eta\in W^1(E)$ and $x\in X$,
	\begin{equation}\label{equivgn}
		C_1\|(D_{x_0}+i)\eta\|_{L^2(E)}\leq \|(D_{x}+i)\eta\|_{L^2(E)}\leq C_2\|(D_{x_0}+i)\eta\|_{L^2(E)}.
	\end{equation}
		 Let $s\in \cs$ and note that, 
	by the left hand-side of (\ref{equivgn}),
	\begin{equation*}
		\|(D_{x_0}+i)\|_{\mathcal{B}(W^1_x(E),L^2(E))}\leq C_1^{-1}.
	\end{equation*}	
 It then follows that 
	\begin{equation*}
	\begin{split}
		\|s\|_{\ee^1_{x_0}(E)}&= \|(D_{x_0}+i)s\|_{\ee^0(E)}\\
		&\leq C_{\co}\cdot\|(D_{x_0}+i)\|_{\mathcal{B}(W^1_x(E),L^2(E))}\|s\|_{\ee^1_x(E)}\\
		&\leq C_{\co}\cdot C_1^{-1}\|s\|_{\ee^1_x(E)},
		\end{split}
	\end{equation*}
	where $C_{\co}$ comes from Proposition \ref{haoguo35}. The other inequality is proved similarly.
\end{proof}

\begin{lemma}\label{a2lemma}
	The map $D\colon X\to \mathcal{L}(\ee^1(E),\ee^0(E))$ is norm-continuous.
\end{lemma}
\begin{proof}
Let $C_{\co}$ be the constant in Proposition \ref{haoguo35}. The result follows by noting that
\begin{equation*}
	\|D_x-D_y\|_{\mathcal{L}(\ee^1(E),\ee^0(E))}\leq C_{\co}\cdot \|D_x-D_y\|_{\mathcal{B}(W^1(E),L^2(E))},
\end{equation*}
for every $x,y\in X$.
	\end{proof} 
	\begin{lemma}\label{a3lemma}
		The operator $D_x$ is invertible on $\ee^0(E)$ for every $x\in X\backslash K$, and $\sup_{x\in X\backslash K}\|D_x^{-1}\|_{\ee^0(E)}<\infty$.
	\end{lemma}
\begin{proof}
By Assumption \ref{assump3} the operator $\kappa_x-c\in \End^G(E)$ is fiberwise positive for each $x\in X\backslash K$, so by Lemma \ref{positivitynoncocomp} it follows that $\kappa_x-c$ defines a positive operator on $\ee^0(E)$. We note that $L_x^*L_x$ is automatically positive in $\ee^0(E)$, since for every $s\in \cs$ 
\begin{equation*}
		\langle (L_x^*L_x) s,s\rangle_{\ee^0(E)}=\langle L_xs,L_xs\rangle_{\ee^0(E')}\geq 0
\end{equation*}
on $\CS(G)$, by $G$-equivariance of the operator $L_x\colon \cs\to \Gamma_c^{\infty}(E')$. It follows that $D_x^2=L_x^*L_x+\kappa_x\geq c > 0$ on $\ee^0(E)$. By \cite[Proposition 1.21]{ebert} we conclude that $D_x^2$ is invertible and 
	\begin{equation*}
		\displaystyle\sup_{x\in X/K} \|D_x^{-1}\|_{\mathcal{L}(\ee^0(E))}\leq 1/\sqrt{c}<\infty.
	\end{equation*}
\end{proof}

\begin{proposition}\label{koen32}
The family $\{D_x\}_{x \in X}$ defines a regular self-adjoint operator $\DD$ on the Hilbert $C_0(X,\CS(G))$-module $C_0(X,\ee^0(E))$ given by (\ref{eqopdd}) on the initial domain $C_c(X,\ee^1(E))$. Furthermore, if the graph norms of $\DF$ are uniformly equivalent, then the closure of $\DD$ is regular self-adjoint on the domain $C_0(X,\ee^1(E))$.
\end{proposition}
\begin{proof}
By Lemmas \ref{a1lemma}, \ref{a2lemma} and \ref{a3lemma}, the family $\{D_x\colon \ee^1(E)\to \ee^0(E)\}_{x\in X}$ satisfies conditions (a1)-(a3) in \cite[\S3.1]{koen}. Hence the claim follows by \cite[Lemma 3.2]{koen}.
\end{proof}

Proposition \ref{prop D dot Fred} now follows from \cite[Proposition 3.4]{koen}. Indeed, by \cite[Lemma 3.3]{koen} the operator $\DD$ has locally compact resolvents. Then, picking a compactly supported function $f\in C_c(X)$ with $f|_K=1$, one may directly check that
		\begin{equation*}
		Q:=(\DD-i)^{-1}f+\DD^{-1}(1-f)
	\end{equation*} 
	defines a parametrix for $\DD$.

The class (\ref{eqclass}) also has a description in terms of unbounded $K\!K$-theory, which shall be useful for later (cf. \cite[Lemma 3.7]{koen} and \cite[Proposition 8.7]{kl1}).
\begin{proposition}\label{a123}
	  Let $\psi\in C_0(X)$ be a strictly positive function vanishing at infinity such that $\psi|_K=1$. Then the operator $\DD'$ determined by the family $\{D'_x:=\psi^{-1}(x)D_x\}_{x\in X}$ defines an odd unbounded Kasparov $(\com,C_0(X,\CS(G))$-cycle $(C_0(X,\ee^0(E)),\DD',1_{\com})$.
\end{proposition}
\begin{proof}
	 Since the family $\{D_x'\}_{x\in X}$ also satisfies Assumptions \ref{assump1}-\ref{assump3}, we see by Proposition \ref{koen32} that the operator $\DD'$ is regular self-adjoint on $C_0(X,\ee^0(E))$. Since the inclusion $\ee^1(E)\hookrightarrow\ee^0(E)$ is compact, the operator $(D_x\pm i\cdot\psi(x))^{-1}$ is compact for every $x\in X$, so that $(\DD\pm i\cdot\psi)^{-1}\in C(X,\mathcal{K}(\ee^0(E)))$. In fact, by Lemma \ref{a3lemma} one has $(\DD\pm i\cdot \psi)^{-1}\in C_b(X,\mathcal{K}(\ee^0(E)))$. It then follows that $(\DD'\pm i)^{-1}=\psi\cdot(\DD\pm i\cdot\psi)^{-1}\in C_0(X,\mathcal{K}(\ee^0(E)))\simeq \mathcal{K}(C_0(X,\ee^0(E)))$, so that $\DD'$ has compact resolvents.
\end{proof}

	\begin{proof}[Proof of Proposition \ref{prop k equals 0}]
 The endomorphism $\kappa_x-c$ is fibrewise positive and, hence, by Lemma \ref{positivitynoncocomp}, we have $\kappa_x-c\geq 0$ on $\ee^0(E)$, for each $x
 \in X$. By the same argument as in the proof of Lemma \ref{a3lemma}, we see that the operator $D_x^2\geq c>0$ is positive on $\ee^0(E)$ and, thus, so is $\DD^2\geq c>0$ on $C_0(X,\ee^0(E))$. It follows from \cite[Proposition 1.21]{ebert} that $\DD$ is an invertible operator and, hence, the class $[D_\bullet]=\SF_G(\DD)\in KK^1(\com,C_0(X,\CS(G))$ vanishes.
 \end{proof}

\subsection{Integrating $G$-Spectral Flow}\label{subsecintgspecflow}

We note that for each fixed $s\in\cs$ and $m\in M$, the map $g\mapsto g(s)(m)\in E_m$ has compact support, by properness of the action. We can thus define the $G$-average of a section $s\in\cs$ by 
\begin{equation*}
    \Phi(s)(m):=\int_G g(s)(m)\,dg\in E_m,
\end{equation*}
for each $m\in M$. We define a map $\mathcal{F}_0\colon \cs\to L^2_T(E)^G$ by setting
\begin{equation*}
    \mathcal{F}_0(s)=\co\Phi(s)\in \co\gi.
\end{equation*}
We define the Hilbert module $\ee^0(E)$ with respect to the maximal group $C^*$-algebra $\CS_{\text{max}}(G)$, and consider the balanced tensor product $\ee^0(E)\otimes_{\phi}\com$, which we shall denote from now on by $\ee_{\phi}^0(E)$. We perceive it as the quotient of the vector space $\ee^0(E)$ by the subspace $N_{\phi}:=\{e\in \ee^0(E): \langle e,e \rangle_{\phi}=0\}$, which is then equipped with the bilinear product

\begin{equation}\label{phiip}
	\langle e_1,e_2\rangle_{\phi}:=\phi(\langle e_1,e_2\rangle_{\ee^0(E)}),
\end{equation}  
defined for each $e_1,e_2\in\ee^0(E)$. The Hilbert space $\ee^0_\phi(E)$ is the completion of this quotient with respect to the norm
\begin{equation*}
	\|e\|^2_{\phi}:=\langle e,e\rangle_{\phi},
\end{equation*}
for each $e\in \ee^0(E)$.

\begin{lemma}\label{f0 unitary}
	For every $s_1,s_2\in\cs$, there holds
	\begin{equation}\label{eq f0 unitary}
		\langle \mathcal{F}_0(s_1),\mathcal{F}_0(s_2)\rangle_{L^2_T(E)^G}=\langle s_1,s_2\rangle_\phi.
	\end{equation}
	\end{lemma}
\begin{proof}
By direct calculation we see that
\begin{equation}\label{eq1}
    \begin{split}
        \langle \mathcal{F}_0(s_1),\mathcal{F}_0(s_2)\rangle_{L^2_T(E)^G}
        &=\int_{M}\co(m)^2\int_G\int_G\langle g_1(s_1)(m),g_2(s_2)(m)\rangle_{E_m}\,dg_1\,dg_2\, d\mu(m)\\
        &=\int_{M}\co(m)^2\int_G\int_G\langle (s_1(g_1^{-1}m),g_1^{-1}(g_2(s_2)(m))\rangle_{E_{g_1^{-1}m}}\,dg_1\,dg_2\, d\mu(m)\\
        &=\int_{M}\co(m)^2\int_G\int_G\langle (s_1(g_1^{-1}m),h(s_2)(g_1^{-1}m)\rangle_{E_{g_1^{-1}m}}\,dg_1\,dh \,d\mu(m),
        \end{split}\end{equation}
        where $h=g_1^{-1}g_2$. It follows that (\ref{eq1}) is equal to
        \begin{equation*}
        \begin{split}
        \int_{M}\int_G\langle (s_1(m),h(s_2)(m)\rangle_{E_{m}}\,dh\, d\mu(m)&=\int_G\langle s_1,h(s_2)\rangle_{L^2}\,dh\\
        &=\langle s_1,s_2\rangle_{\phi}
    \end{split}
\end{equation*}
where we used the fact that $d\mu$ is $G$-invariant and that $G$ is unimodular.
\end{proof}

By Lemma \ref{f0 unitary} the map $\mathcal{F}_0$ extends to a bounded operator
	\begin{equation*}
		\overline{\mathcal{F}_0}\colon \ee^0(E)\to L^2_T(E)^G,
	\end{equation*}
since, for every $s\in\cs$, there holds
\begin{equation}\label{eq2}
	\begin{split}
		\|\mathcal{F}_0(s)\|_{L^2_T(E)^G}=\|s\|_{\phi}\leq\|s\|_{\ee^0(E)}.
	\end{split}
\end{equation}  
Inequality (\ref{eq2}) above is a consequence of $\|\cdot\|_{\phi}$ being the norm on $C_c(G)$ associated with the trivial representation of $G$. This fact is  the reason why we consider the maximal group $C^*$-algebra $\CS_{\rm{max}}(G)$ in this section.
\begin{lemma}\label{lemma intertwine e1}
	For every $x\in X$ and every $e_1\in \ee^1(E)$, there holds $\overline{\mathcal{F}_0}(e_1)\in W^1_{T,x}(E)^G$ and
	\begin{equation*}
		\overline{\mathcal{F}_0}\circ D_x(e_1)=\widetilde{D}_x\circ \overline{\mathcal{F}_0}(e_1).
	\end{equation*}
\end{lemma}
\begin{proof}
	We first prove the claim for compactly supported smooth sections $s\in \cs$. Let $m\in M$ and suppose $s\in\cs$ is supported in a coordinate chart $(U;x^1,\ldots,x^n)$, with $U$ relatively compact. Let $\chi\in C_c^{\infty}(M)$ be a compactly supported smooth function such that $\chi|_U=1$. We see that $[\partial_i,\chi]|_U=0$ and, hence, that
\begin{equation*}
	\begin{split}
\left( \partial_i\int_G g(s)\,dg \right )(m)&=\left(\chi\partial_i\int_G g(s)\,dg\right)(m)\\
&=\left(\partial_i\int_G \chi g(s)\,dg\right)(m).
	\end{split}
\end{equation*}
By properness of the action, the map $(m',g)\mapsto (\chi gs)(m')$ has compact support in $M\times G$ and hence its derivatives are also compactly supported. We may thus write
\begin{equation*}
	\begin{split}
		\left( \partial_i\int_G g(s)\,dg \right )(m)&=\int_G\partial_i \chi g(s)(m)\,dg\\
		&=\int_G\partial_i (g(s))(m)\,dg.
	\end{split}
\end{equation*}
It follows that $D_x\circ\Phi(s)=\Phi\circ D_x(s)$, which we can use to show that
\begin{equation*}
    \begin{split}
        \widetilde{D}_x\circ \overline{\mathcal{F}_0}(s)&=\widetilde{D}_x(\co\Phi(s))\\
        &=\co(D_x\circ\Phi)(s)\\
        &=\co(\Phi\circ D_x)(s)\\
        &=\overline{\mathcal{F}_0}\circ D_x(s).
    \end{split}
\end{equation*}
 We now see that for every $s\in \cs$, 
\begin{equation*}
\begin{split}
	\|(\widetilde{D}_x+i)(\overline{\mathcal{F}_0}(s))\|_{L^2_T(E)^G}&= \|\overline{\mathcal{F}_0}(D_x+i)(s)\|_{L^2_T(E)^G}\\
	&\leq \| (D+i)(s)\|_{\ee^0(E)} ,
	\end{split}
\end{equation*}
where we used (\ref{eq2}). Hence, it follows that if $e_1\in \ee^1(E)$, then $\overline{\mathcal{F}_0}(e_1)\in W^1_{T,x}(E)^G$, which concludes the proof.
\end{proof}
It follows from (\ref{eq f0 unitary}) that $\overline{\mathcal{F}}_0(N_{\phi})=\{0\}$, and hence $\overline{\mathcal{F}}_0$ induces a well-defined isometry between $\ee^0_\phi(E)$ and $L^2_T(E)^G$, which we denote by
 \begin{equation*}
 	\mathcal{F}_1\colon \ee^0_\phi(E)\to L^2_T(E)^G.
 \end{equation*}
We note that for each $s\in \Gamma^{\infty}(E)^G$ there holds $\mathcal{F}_1(\co^2s)=\co s$, which proves that $\mathcal{F}_1$ is surjective, and hence a unitary isomorphism. The map
\begin{equation*}
	\mathcal{F}\colon C_0(X,\ee^0_\phi(E))\to C_0(X,L^2_T(E)^G),
\end{equation*}
defined as $\mathcal{F}(\eta_{\phi})(x):=\mathcal{F}_1(\eta_{\phi}(x))$, for every $\eta_{\phi}\in C_0(X,\ee^0_\phi(E))$,  is then easily seen to define a unitary isomorphism.
\begin{remark}\label{remark d tilde is symmetric}
  We note here that the regular self-adjoint operator $D_x$ on $\ee^0(E)$, with domain $\ee^1(E)$ defines a self-adjoint operator $\phi_*(D_x)\colon \Dom(\phi_*(D_x))\to \ee^0_\phi(E)$ (\cite[\S9]{lance}). As a result of Lemma \ref{lemma intertwine e1}, the map $\mathcal{F}_1$ intertwines $\phi_*(D_x)$ and $\widetilde{D}_x$, which implies that $\widetilde{D}_x$ is self-adjoint on $W^1_{T,x}(E)^G$. We can also check directly that $\widetilde{D}_x$ is symmetric: let $s_1, s_2\in \Gamma^{\infty}(E)^G$ and note that
 \begin{equation}\label{eqn d tilde symmetric}
 	\langle \widetilde{D}_x\co s_1,\co s_2\rangle_{L^2_T(E)^G}=\langle \co s_1,\widetilde{D}_x\co s_2\rangle_{L^2_T(E)^G}-2\langle [D_x,\co]s_1,\co s_2\rangle_{L^2(E)}.
 \end{equation}
 Calculating the second term on the right-hand side of (\ref{eqn d tilde symmetric}), we get
 \begin{equation}\label{eqn 53}
 		\langle [D_x,\co ]s_1,\co s_2\rangle_{L^2(E)}=\int_M \co(m)^2\int_G\langle \co(gm)[D_x,\co](gm)s_1(gm),s_2(gm)\rangle_{E_{gm}}\,dg\, d\mu(m).
 \end{equation}
 We note that 
 \begin{equation*}
 	\begin{split}
( 		\co[D_x,\co](s_1))(gm)&=\left(\frac{1}{2}[D_x,\co^2](s_1)(gm)\right)\\
&=g\cdot\left(\frac{1}{2}[D_x,g^{-1}(\co^2)](s_1)(m)\right),
 	\end{split}
 \end{equation*}
 where we used $G$-equivariance of $D_x$ and $G$-invariance of $s_1$. Putting it all back together into (\ref{eqn 53}), and using $G$-invariance of $s_2$, we can write
 \begin{equation*}
 	\begin{split}
 		\langle [D_x,\co]s_1,\co s_2\rangle_{L^2(E)}&=\frac{1}{2}\int_M \co(m)^2\int_G\langle g\cdot([D_x,g^{-1}(\co^2)]s_1(m)),g\cdot s_2(m)\rangle_{E_{gm}}\,dg\,d\mu(m)\\
 		&=\int_M \co(m)^2\left\langle \left[D_x,\left(\int_G g^{-1}(\co^2)\,dg\right)\right]s_1(m),s_2(m)\right\rangle_{E_m}\,d\mu(m)\\
 		&=0,
 	\end{split}
 \end{equation*}
 where we used that $G$ acts by isometries on the fibers, and that $\left(\int_G g^{-1}(\co^2)\,dg\right)(m)=1$ for every $m\in M$. \end{remark}

\begin{proof}[Proof of Proposition \ref{intgspecflo}]
Let $\widetilde{\phi}:=1\otimes\phi\colon C_0(X)\otimes\CS_{\rm{max}}(G)\to C_0(X)$. By \cite[Lemma 3.7]{koen}, $\SF_G(\DD)$ can be represented by the unbounded Kasparov $(\com,C_0(X)\otimes\CS_{\rm{max}}(G))$-cycle $(C_0(X,\ee^0(E)),\psi^{-1}\DD, 1_{\com})$ (cf. Proposition \ref{a123} and Remark \ref{remark unb sf}). We have
\begin{equation}\label{intgspecfloeq}
    \begin{split}
        \SF_G(D_\bullet)\otimes_{\CS_{\rm{max}}(G)}[\phi_*]&=[(C_0(X,\ee^0(E)),\psi^{-1}\DD, 1_{\com})]\otimes_{\CS_{\text{max}}(G)}[(\com,0,\phi)]\\
        &=[(C_0(X,\ee^0(E)),\psi^{-1}\DD, 1_{\com})]\otimes_{C_0(X)\otimes\CS_{\text{max}}(G)}[(C_0(X),0,\widetilde{\phi})]\\
        &=[(C_0(X,\ee^0(E))\otimes_{\widetilde{\phi}}C_0(X),\psi^{-1}\DD\otimes 1, 1_{\com})],
    \end{split}
\end{equation}
where we use \cite[Lemma 2.8]{koen} for the third equality. We define a map
\begin{equation*}
	\widetilde{\mathcal{F}}\colon C_0(X,\ee^0(E))\otimes_{\widetilde{\phi}}C_0(X)\to C_0(X,L^2_T(E)^G),
\end{equation*}
given by the composition
\begin{equation*}
	\widetilde{\mathcal{F}}\colon C_0(X,\ee^0(E))\otimes_{\widetilde{\phi}}C_0(X)\simeq C_0(X,\ee^0_{\phi}(E))\overset{\mathcal{F}}\simeq C_0(X,L^2_T(E)^G),
\end{equation*}
which is a unitary isomorphism. The result now follows from Lemma \ref{lemma intertwine e1}, which can be used to see that $\widetilde{\mathcal{F}}$ intertwines $\psi^{-1}\DD\otimes 1$ and $\psi^{-1}\widetilde{D}_\bullet$.
\end{proof}

We now turn to the universal cover case (Example \ref{univcover}), and consider a family $\{\widehat{D}_x\colon \Gamma_c^{\infty}(\widehat{E})\to L^2(\widehat{E})\}_{x\in X}$ on a compact manifold $\widehat{M}$ satisfying Assumptions \ref{assump2} and \ref{assump3}. We prove that the lifted family $\{D_x\colon\cs\to L^2(E)\}_{x\in X}$ of $\Gamma$-invariant operators on the universal cover $M$ satisfies the necessary assumptions so that it has an equivariant spectral flow. 

Let $U\subset M$ be a fundamental domain for the $\Gamma$-action on $M$. This means that $U$ is a relatively compact open set such that $\gamma U\cap U = \emptyset$ for every $\gamma\neq e$, and that $M\backslash(\cup_{\gamma\in\Gamma} \gamma U)$ has measure zero.

\begin{lemma}\label{lemma lift assump univ cov}
	The family $\{D_x\colon \cs\to L^2(E)\}_{x\in X}$ satisfies Assumptions \ref{assump1}-\ref{assump3}.
\end{lemma}
\begin{proof}
	We note that Assumption \ref{assump1} is readily satisfied because of Proposition \ref{prop garding g discrete}. Let $s\in \cs$ and note that 
	\begin{equation*}
		\begin{split}
			\|(D_x -D_y)s\|_{L^2(E)}&=\sum_{\gamma\in \Gamma}\| (D_x-D_y)s|_{\gamma U}\|_{L^2(E)}\\
			&=\sum_{\gamma\in \Gamma}\|(D_x-D_y)\gamma (s)|_{\gamma U}\|_{L^2(E)},
		\end{split}
	\end{equation*}
for every $x,y\in X$, where we use $\Gamma$-equivariance of $D_x-D_y$. We note that $\gamma(s)|_{\gamma U}$ is supported within $U$. It follows that
	\begin{equation*}
		\begin{split}
			\|(D_x -D_y)s\|_{L^2(E)}&\leq \|(D_x-D_y)|_U\|_{\mathcal{B}(W^1(E|_U),L^2(E|_U))}\cdot \sum_{\gamma\in \Gamma} \|\gamma(s)|_{\gamma U}\|_{W^1(E)}\\
			&=\|(D_x-D_y)|_U\|_{\mathcal{B}(W^1(E|_U),L^2(E|_U))}\|s\|_{W^1(E)}\\
			&=\| \widehat{D}_x - \widehat{D}_y\|_{\mathcal{B}(W^1(\widehat{E}),L^2(\widehat{E}))}\|s\|_{W^1(E)}.
		\end{split}
	\end{equation*}
	We thus see that $\|D_x-D_y\|_{\mathcal{B}(W^1(E),L^2(E))}\leq \|\widehat{D}_x-\widehat{D}_y\|_{\mathcal{B}(W^1(\widehat{E}),L^2(\widehat{E}))}$. The result then follows from the norm-continuity of $x\mapsto \widehat{D}_x$ as a map $X\to \mathcal{B}(W^1(\widehat{E}),L^2(\widehat{E}))$.
	
	Now, we know that $\widehat{D}_x=\widehat{L}_x^*\widehat{L}_x+\widehat{\kappa}_x$ for every $x\in X\backslash K$, where $\widehat{L}_x$ is a differential operator and $\widehat{\kappa}_x$ is an endomorphism. These can be lifted to, respectively, a $\Gamma$-equivariant differential operator $L_x$ and $\Gamma$-invariant endomorphism $\kappa_x$ on $E\to M$, and there holds $D_x^2=L_x^*L_x+\kappa_x$ for every $x\in X\backslash K$. Let $c>0$ such that $\widehat{\kappa}_x\geq c$ on $L^2(\widehat{E})$. We note that $\kappa_x\geq c>0$ fiberwise, by construction. 
	
	\end{proof}

\begin{proof}[Proof of Corollary \ref{corollary summation map}]
	The family $\DF$ satisfies Assumptions \ref{assump1}-\ref{assump3}, by Lemma \ref{lemma lift assump univ cov}. Hence, it defines a $K\!K$-theoretic $G$-spectral flow class $\SF_G(\DD)\in K\!K^1(\com,C_0(X,\CS_{\rm{max}}(G))$. Since the action of $\Gamma$ on $M$ is free, we can see by the results in  \cite[\S4.3]{petermathai} that there exist a unitary isomorphism $L^2_T(E)^{\Gamma}\simeq L^2(\widehat{E})$ which intertwines the operators $\widetilde{D}_x$ and $\widehat{D}_x$, for every $x\in X$. It is then straightforward to see that the unbounded Kasparov $(\com, C_0(X))$-cycle $(C_0(X,L^2_T(E)^{\Gamma}),\psi^{-1}\widetilde{D}_\bullet, 1_{\com})$ represents the class $\SF_{\{e\}}(\widehat{D}_\bullet)$ (see Remark \ref{remark unb sf}). The result now follows from Proposition \ref{intgspecflo}.
\end{proof}

\section{Index Theory and Spectral Flow}\label{seceqindextheory}

This section contains the proofs of our ``index equals spectral flow'' results, Theorems \ref{propfredindeqsfg1} and \ref{propindtheorem}.
We continue with the notation and assumptions from Subsection \ref{sec SFG}.

\subsection{Continuous Families}\label{subseccontfam}

In this subsection we also assume that the family $\DF$ is locally constant outside of $K\subset X$ (Assumption \ref{assump4}). 
Define the \textit{product operator}
\begin{equation}\label{product operator}
    \widetilde{P}_{D}:=\begin{pmatrix}0&P_X+i\DD\\P_X-i\DD&0
    \end{pmatrix}
\end{equation}
on $(\text{Dom}(P_X)\cap \text{Dom}(\DD))^{\oplus 2} \subset (\ee^0(E)\otimes L^2(X,F))^{\oplus 2}$. 
It follows from \cite[Theorem 4.3]{koen} that the operator $\widetilde{P}_D$ is regular self-adjoint and Fredholm on $(\ee^0(E)\otimes L^2(X,F))^{\oplus 2}$ and, hence, it defines a class
\begin{equation}\label{eq KK PD}
	[\widetilde{P}_D]\in K\!K^0(\com,\CS(G)),
\end{equation}
which is mapped to the Fredholm index 
\begin{equation*}
	\Index(P_X-i\DD)\in K_0(\CS(G))
\end{equation*}
 of $P_X-i\DD$ on $\ee^0(E)\otimes L^2(X,F)$ under the isomorphism $K\!K^0(\com,\CS(G))\simeq K_0(\CS(G))$ (see e.g. Subsection \ref{subsecfredind}). The following version of Theorem \ref{propfredindeqsfg1} is a direct consequence of  \cite[Theorem 5.15]{koen}:
\begin{theorem}\label{propfredindeqsfg}
	Suppose the family $\DF$ satisfies Assumptions \ref{assump1}-\ref{assump4}. Then the the class $[\widetilde{P}_D]\in K\!K^0(\com,\CS(G))$ is given by the Kasparov product between the $K\!K$-theoretic $G$-spectral flow class $\SF_G(\DD)\in K\!K^1(\com,C_0(X,\CS(G)))$ and the $K$-homology class $[P_X]\in K\!K^1(C_0(X),\com)$. In other words,
\begin{equation*}
	\Index(P_X-i\DD)=\SF_G(\DD)\otimes_{C_0(X)}[P_X]\in K_0(\CS(G)).
\end{equation*}
In particular, if $X=\re$ then
\begin{equation*}
	\Index(-i\partial_t-i\DD)=\Sf_G(\DD)\in K_0(\CS(G)).
\end{equation*}
\end{theorem}
\begin{proof}
It is straightforward to check that if the family $\DF$ is locally constant outside of $K$, then it satisfies Assumption (A4') in \cite{koen}.  The result hence follows from \cite[Theorem 5.15]{koen}.
\end{proof}

\subsection{Differentiable Families}\label{subsecdiffam}

We prepare for the proof of Theorem \ref{propindtheorem} by proving some properties of weak derivatives.
We assume throughout this subsection that the family $D\colon X\to \mathcal{B}(W^1(E),L^2(E))$ has uniformly bounded weak derivative (see Definition \ref{defweakdif}). For each $x\in X$, we write $\|dD(x)\|_{\mathcal{B}(W^1(E),L^2(E)\otimes T_x^*X)}$ simply as $\|dD(x)\|$ in order to simplify the notation.

\begin{lemma}\label{lemma weak der local}
	For each $x\in X$, the weak derivative $dD(x)\in \mathcal{B}(W^1(E),L^2(E)\otimes T_x^*X)$ can be restricted to a local, $G$-equivariant operator $dD(x)\colon \cs\to\cs\otimes T_x^*X$ between smooth sections.
\end{lemma}
\begin{proof}
	Let $s_1\in\cs$ and assume $s_1|_U=0$ for some open subset $U\subset M$. Let $s_2\in\cs$ be a section supported within $U$. We see that
		\begin{equation}\label{eq proof weak dif}
	 	\lim_{h\to 0}\frac{\left |\langle (dD(x)(h)s_1)|_U,s_2|_U\rangle_{L^2(E)} \right|}{\|h\|_{T_xX}}=0,
	\end{equation}
	for every $h\in T_x X$. Since $h\mapsto \langle (dD(x)(h)s_1)|_U,s_2|_U\rangle_{L^2}$ is a linear map, it follows from (\ref{eq proof weak dif}) that 
$$(dD(x)(h)s_1)|_U=0,$$ 
for every $h\in T_x^*X$ and $x\in X$, which proves the locality condition. The fact that $dD(x)$ is $G$-equivariant is similarly obtained by calculating that
\begin{equation*}
	\lim_{h\to 0}\frac{\left|\left\langle \left(g(dD(x)(h))-dD(x)(h))s_1,s_2\right)\right\rangle_{L^2(E)}\right|}{\|h\|_{T_xX}}=0
\end{equation*}
for every $x\in X$ and $s_1,s_2\in \cs$, which follows by $G$-equivariance of $D_x$. The fact that $dD(x)(h)$ maps smooth sections to smooth sections is a consequence of its dependence on the local coefficients of $D_x$, which are smooth.
\end{proof}
\begin{lemma}\label{lemma e0 weak derivatives}
	 For each $x\in X$, the weak derivative $dD(x)\colon \cs\to\cs\otimes T_x^*X$ extends to an adjointable operator $dD(x)\in \mathcal{L}(\ee^1(E),\ee^0(E)\otimes T_x^*X)$. Moreover,
	\begin{equation}\label{eq unif boun weak deri e0}
		\displaystyle\sup_{y\in X}\|dD(y)\|_{\mathcal{L}(\ee^1(E),\ee^0(E)\otimes T_x^* X)}< \infty.
	\end{equation}
\end{lemma}
\begin{proof}
	By Lemma \ref{lemma weak der local} we can apply Proposition \ref{haoguo35} in order to conclude that $dD(x)\in \mathcal{L}(\ee^1(E),\ee^0(E)\otimes T_x^*X)$. Furthermore, there exists a constant $C_{\co}$ such that
	\begin{equation*}
		\|dD(x)\|_{\mathcal{L}(\ee^1(E),\ee^0(E)\otimes T_x^*X)}\leq C_{\co}\cdot\|dD(x)\|_{\mathcal{B}(W^1(E),L^2(E)\otimes T_x^*X)},
	\end{equation*}
for every $x\in X$, from which we can conclude (\ref{eq unif boun weak deri e0}).
\end{proof}

\begin{lemma}\label{difl2dife0}
For each pair of sections $s_1, s_2\in \cs$, the map $x\mapsto \langle D_x s_1,s_2\rangle_{\ee^0(E)}$ is differentiable. 
\end{lemma}
\begin{proof} 
Fix $x\in X$ and let $dD(x)\colon \cs\to\cs\otimes T_x^*X$ be as in Lemma \ref{lemma weak der local}. Note that by (\ref{l1ineq}), it is enough to prove that
\begin{equation}\label{eq lim int dif}
		\displaystyle\lim_{n\to \infty}\int_G\frac{\left|\left\langle \left(D_{x+h/n}-D_x-dD(x)(h/n)\right)s_1,g(s_2)\right\rangle_{L^2(E)}\right|}{\|h/n\|_{T_xX}}\,dg=0,
\end{equation}
for every $h\in T_x^*X$. Let $f_n\colon G\to \re$ be defined as the integrand 
\begin{equation*}
	f_n(g):=\frac{\left|\left\langle \left(D_{x+h/n}-D_x-dD(x)(h/n)\right)s_1,g(s_2)\right\rangle_{L^2(E)}\right|}{\|h/n\|_{T_xX}}.
\end{equation*}
By hypothesis, we have that $f_n\to 0$ pointwise. The set 
\begin{equation*}
 	K_0:=\{g\in G: \text{supp}(s_1)\cap \text{supp}(g(s_2))\neq \emptyset\}
 \end{equation*}
is compact by properness of the action, and the functions $f_n$ are supported in $K_0$, for every $n\in \mathbb{N}$, since the operator $D_x-D_{x+h}-dD(x)(h/n)$ is local by Lemma \ref{lemma weak der local}. We can see from (\ref{contfromdiff}) that
\begin{equation}\label{eq bound fn}
	\left|\left\langle \left(D_{x+h/n}-D_x\right)s_1,g(s_2)\right\rangle_{L^2(E)}\right|\leq \displaystyle\sup_{y\in X} \|dD(y)\|\cdot \|s_1\|_{W^1(E)}\cdot\|s_2\|_{L^2(E)}\cdot \|h/n\|_{T_xX},\end{equation}
where we use the fact that $G$ acts by isometries, along with the fact that 
\begin{equation*}
	\begin{split}
		\text{dist}(x,x+h/n)&=\text{dist}(x,\exp_x(h/n))\\ 
		&=\|h/n\|_{T_xX}.
	\end{split}
\end{equation*}
The right-hand side of (\ref{eq bound fn}) is finite, since the number $\sup_{y\in X}\|dD(y)\|$ is finite by assumption. We can estimate
\begin{equation*}
\begin{split}
		f_n(g)&\leq  \frac{\left|\left\langle \left(D_{x+h/n}-D_x\right)s_1,g(s_2)\right\rangle_{L^2(E)}\right|+\left|\left\langle dD(x)(h/n)s_1,g(s_2)\right\rangle_{L^2(E)}\right|}{\|h/n\|_{T_xX}}\\
		&\leq  2\cdot\displaystyle\sup_{y\in X}\|dD(y)\|\cdot\|s_1\|_{W^1(E)}\cdot \|s_2\|_{L^2(E)}
		\end{split}
\end{equation*}
for every $n\in \mathbb{N}$ and $g\in G$. Set 
\begin{equation*}
 	V:=2\cdot\displaystyle\sup_{y\in X}\|dD(y)\|\cdot\|s_1\|_{W^1(E)}\cdot \|s_2\|_{L^2(E)}
 \end{equation*}
 and let $l\colon G\to \re$ be a compactly supported continuous function such that $l|_{K_0}=V$. It follows that $l$ is an integrable function on $G$ such that $|f_n(g)|\leq l(g)$, for every $n\in\mathbb{N}$ and $g\in G$. Then, (\ref{eq lim int dif}) follows from the fact that $f_n\to 0$ pointwise and the Dominated Convergence Theorem.
\end{proof}

\begin{remark}
Our results do not allow us to state that the map $D\colon X\to \mathcal{L}(\ee^0(E),\ee^1(E))$ has uniformly bounded weak derivative, because we can only prove Lemma \ref{difl2dife0} for compactly supported smooth sections $s_1,s_2\in \cs$. However, we shall see in Subsection \ref{subsec estimates com} that this is enough for us to derive the ``index equals spectral flow" Theorem \ref{propindtheorem}.
\end{remark}

\subsection{Estimates for Commutators}\label{subsec estimates com}

In order to prove Theorem \ref{propindtheorem}, we analyse the commutator between $P_X$ and $\DD'$, as follows:
\begin{definition}[{\rm{\cite[Assumption 7.1]{KL2}}}]\label{klass7.1}
	Let $P$ and $S$ be regular self-adjoint operators on a Hilbert $B$-module $\mathcal{E}$, and let $\mu\in \re\backslash\{0\}$. One says that $[P,S](S-i\mu)^{-1}$ is \textit{well-defined and bounded} on $\mathcal{E}$ when
	\begin{enumerate}[(a)]
		\item There exists a submodule $\mathcal{S}\subset \ee$ which is a core for $P$.
		\item\label{cond2wdb} The following inclusions hold:
		\begin{equation}
			(S-i\mu)^{-1}(\xi)\in \text{Dom}(P)\cap \text{Dom}(S)\;\;\;\text{and}\;\;\; P(S-i\mu)^{-1}(\xi)\in \text{Dom}(S)
		\end{equation}
		for all $\xi \in \mathcal{S}$.
		\item The map
		\begin{equation*}
			[P,S](S-i\mu)^{-1}\colon\mathcal{S}\to\ee
		\end{equation*}
		extends to a bounded, adjointable operator in $\mathcal{L}(\ee)$.	\end{enumerate}
\end{definition}

\begin{lemma}\label{welldefinedbounded}
	Let $P_X$ and $\DF$ be as in Theorem \ref{propindtheorem}. 
	Then $[P_X,\DD](\DD\pm i\mu)^{-1}$ is well-defined and bounded on $\ee^0(E)\otimes L^2(X, F)$, for every $\mu\in\re\backslash\{0\}$.\end{lemma}
\begin{proof}
The submodule $\widetilde{\mathcal{S}}:=\cs\otimes \Gamma_c^{\infty}(X,F)\subset \ee^0(E)\otimes L^2(X,F)$  can be seen to define a core for $P_X$ on $\ee^0(E)\otimes L^2(X,F)$. By Lemma \ref{lem UEGN}, the induced family $\{D_x\colon \ee^1(E)\to \ee^0(E)\}_{x\in X}$ has uniformly equivalent graph norms on $\ee^0(E)$.	
	We know from Lemma \ref{difl2dife0} that the map $x\mapsto \langle D_x s_1,s_2\rangle_{\ee^0}$ is differentiable for each $s_1, s_2\in \cs$, so that the argument in the proof of \cite[Lemma 8.5]{kl1} can be replicated on the submodule $\widetilde{\mathcal{S}}\simeq \Gamma_c^{\infty}(X,\Gamma_c^{\infty}(M,E)\otimes F)\subset L^2(X,\ee^0(E)\otimes F)$. Namely, for every $\mu\in\re\backslash\{0\}$, condition \ref{cond2wdb} is satisfied, and the operator $[P_X,\DD](\DD-i\mu)^{-1}\colon \widetilde{\mathcal{S}}\to \ee^0(E)\otimes L^2(X,F)$ is given by (cf. \cite[Theorem 8.6]{kl1})
	\begin{equation} \label{eq comm PX D}
		[P_X,\DD](\DD-i\mu)^{-1}=\sigma_{P_X}\cdot d(D(\cdot))\cdot(\DD-i\mu)^{-1}.
	\end{equation}
	One can then see that operator (\ref{eq comm PX D}) extends to a bounded operator on $\ee^0(E)\otimes L^2(X,F)$. Indeed, the operator $d(D(\cdot))\colon \ee^1(E)\otimes L^2(X, F)\to \ee^0(E)\otimes L^2(X, T^*X\otimes F)$ is bounded by Lemma \ref{lemma weak der local}, and the principal symbol $\sigma_{P_X}\colon \ee^0(E)\otimes  L^2(X, T^*X\otimes F)\to\ee^0(E)\otimes  L^2(X,F)$ is bounded because $P_X$ has finite propagation speed (cf. \cite[Proposition 8.2]{kl1}).
\end{proof}

\begin{lemma}\label{lemmaunbktriple}
	Let $P_X$ and $\DF$ be as in Theorem \ref{propindtheorem}. Then there exists a positive, smooth function $\psi\in C^{\infty}(X)$ vanishing at infinity such that the product operator $\widetilde{P}_{D'}$ on $(\ee^0(E)\otimes L^2(X,F))^{\oplus 2}$ is regular self-adjoint on the domain $(\Dom(\DD')\cap\Dom(P_X))^{\oplus 2}$. Moreover, the triple $(\com, (\ee^0(E)\otimes L^2(X,F))^{\oplus 2},\widetilde{P}_{D'})$ defines an unbounded Kasparov $(\com,\CS(G))$-cycle.
\end{lemma}
\begin{proof}
By \cite[Lemma 8.10]{kl1}, there exists a smooth function $\psi\in C^{\infty}(X)$ vanishing at infinity such that $0<\psi(x)\leq 1$ for every $x\in X$, $\psi|_K=1$ and $\|d\psi^{-1}\|_{\infty}<\infty$. It follows then that $[P_X,\DD'](\DD'\pm i\mu)^{-1}$ is well-defined and bounded. Indeed, the estimates in \cite[Lemma 5.7]{koen}, together with the equality
\begin{equation*}
	[P_X,\psi^{-1}\DD]=\sigma_{P_X}(d\psi^{-1})\DD+\psi^{-1}[P_X,\DD],
\end{equation*}
show that the operator $[P_X,\DD'](\DD'\pm i\mu)$ on the submodule $\widetilde{\mathcal{S}}=\Gamma_c^{\infty}(M,E)\otimes\Gamma_c^{\infty}(X,F)$ extends to a bounded operator, for every $\mu\in\re\backslash\{0\}$. The fact that $\widetilde{P}_{D'}$ is regular self-adjoint then follows from \cite[Theorem 7.10]{KL2}. Since $\DD'$ has compact resolvents in $C_0(X,\ee^0(E))$ (cf. Proposition \ref{a123}), we have by \cite[Proposition 4.1]{koen} that $\widetilde{P}_{D'}$ also has compact resolvents on $(\ee^0(E)\otimes L^2(X,F))^{\oplus 2}$ and, hence, that it defines an unbounded Kasparov $(\com,\CS(G))$-cycle, as desired. 
\end{proof}	

\begin{proof}[Proof of Theorem \ref{propindtheorem}]
By Lemma \ref{lemmaunbktriple} we can take $\chi(x)=x(x^2+1)^{-1/2}$ as normalising function for $\widetilde{P}_{D'}$, so that it defines a class
\begin{equation*}
	[\widetilde{P}_{D'}]\in K\!K^0(\com,\CS(G)).
\end{equation*}
We conclude that the operator $P_X-i\DD'$ is Fredholm and, hence, that it defines a class
\begin{equation*}
	\Index(P_X-i\DD')\in K_0(\CS(G)),
\end{equation*}
given by the image of $[\widetilde{P}_{D'}]$ under the isomorphism $K\!K^0(\com,\CS(G))\simeq K_0(\CS(G))$.
The result now follows precisely as in \cite[Proposition 5.10]{koen}, where one shows that the three conditions in Kucerovski's Theorem (\cite[Theorem 13]{kuce}) hold.
\end{proof}

\section{Localised Indices, $\rho$- and $\eta$-Invariants}\label{seclocind}

In order to prove Theorems \ref{thm eta} and \ref{theoremmain}, we relate the Fredholm index from Subsection \ref{subsecfredind} to other constructions of equivariant indices for group actions that are not necessarily cocompact. We apply these arguments to the action by $G$ on $M\times X$. 

To this end, we consider a Riemannian manifold $N$ equipped with a proper, isometric action by a locally compact, unimodular group $G$, and an equivariant $\mathbb{Z}_2$-graded Hermitian vector bundle $S \to N$ with the same structure and properties as $M$ and $E$  at the start of Section \ref{secpreliminaries}, respectively. As in Subsection \ref{subsecgsobolevmod}, we consider a $G$-equivariant, symmetric, elliptic, first-order differential operator $P$ on $S$ which is odd with respect to the grading on $S$, and consider its Sobolev modules $\ee^k(S)$. We assume furthermore that $P$ has finite propagation speed, so that the induced operator $P$ on 
$\ee^0(S)$ with domain $\ee^1(S)$ is regular self-adjoint (\cite[Proposition 5.5]{haoguo}).

\subsection{The Localised $G$-Index}\label{subseclocalgind}

The operator $P$ is said to be \textit{$G$-invertible at infinity} if there exists a non-negative, smooth, $G$-invariant and cocompactly supported function $f$ on $N$ such that $P^2+f\in\mathcal{L}(\ee^2(S),\ee^0(S))$ is invertible, and the inverse satisfies $(P^2+f)^{-1}\in \mathcal{L}(\ee^0(S),\ee^2(S))$. Following \cite[\S4]{haoguo}, given an operator $P$ which is $G$-invertible at infinity, one may define the operator $F:=P(P^2+f)^{-1/2}$ on $\ee^0(S)$, which is given by the integral formula (\cite[Proposition 4.5]{haoguo})
\begin{equation}\label{eqintegralformula}
	F=\frac{2}{\pi}\int_0^{\infty} P(P^2+f+\lambda^2)^{-1}\,d\lambda.
\end{equation}
One then shows that $F^2-1\in \mathcal{K}(\ee^0(S))$, so that $F$ is invertible modulo compacts (\cite[Proposition 4.13]{haoguo}). The short exact sequence
\begin{equation*}
	0\longrightarrow \mathcal{K}(\ee^0(S))\longrightarrow \mathcal{L}(\ee^0(S))\longrightarrow\mathcal{L}(\ee^0(S))/\mathcal{K}(\ee^0(S))\longrightarrow 0
\end{equation*}
induces a six-term exact sequence in $K$-theory, where we denote the relevant boundary maps by $\partial\colon K_i(\mathcal{L}(\ee^0(S))/\mathcal{K}(\ee^0(S))\to K_{i-1}(\mathcal{K}(\ee^0(S))$. Since $F$ is invertible modulo compacts,
it defines a class $[F]\in K_1(\mathcal{L}(\ee^0(S))/\mathcal{K}(\ee^0(S)))$. 

\begin{definition}
The \textit{localised $G$-index} of $P$ is the class
\begin{equation}\label{eq loc g ind}
	\Ind_G(P):=\partial[F]\in K_0(\CS(G))\simeq K_0(\mathcal{K}(\ee^0(S)).
\end{equation}
\end{definition}

The operator $F$ can be shown to be self-adjoint modulo compacts, so that the triple $(\ee^0(S),F,1_{\com})$ defines a Kasparov $(\com,\CS(G))$-cycle and, hence, a class
\begin{equation}\label{eqcalliaskkclass}
	[(\ee^0(S),F,1_{\com})]\in K\!K^0(\com,\CS(G)).
\end{equation}
By \cite[Theorem 4.19]{haoguo}, the class (\ref{eqcalliaskkclass}) is independent of the choice of function $f$ and maps to $\Ind_G(P)$ through the isomorphism $K\!K^0(\com,\CS(G))\simeq K_0(\CS(G))$.

We assume until the end of this subsection that the operator $P$ satisfies
\begin{equation*}
	P^2=\widetilde{\Delta}+A,
\end{equation*}
where $\widetilde{\Delta}$ is a $G$-equivariant differential operator on $S$ which is $\ee^0(S)$-positive, and $A\in\End(S)^G$ is a $G$-equivariant endomorphism with $A\geq c>0$ fiberwise outside a cocompact set $Y\subset N$. We then prove that $P$ is $G$-invertible at infinity. In order to do that we follow \cite[\S5]{haoguo}, where Guo shows that a $G$-Callias operator $D+\Phi$ is $G$-invertible at infinity. The main ingredient of his proof is the assumption that the operator $D\Phi+\Phi D+\Phi^2$ is a $G$-equivariant, $L^2$-positive endomorphism outside a cocompact set, which in our setting corresponds to the positivity of the endomorphism $A$ outside of $Y$. 
\begin{lemma}\label{positivel2diracschro}
	There exists a $G$-invariant, non-negative, smooth, cocompactly supported function $f\in C^{\infty}(N)^G$ and a constant $r>0$ such that $A+f\geq r>0$ on $L^2(N,S)$.\end{lemma}
\begin{proof}
	Let $K\subset Y$ be a compact set such that $G\cdot K=Y$. Then $A$ is bounded from below on $K$ by some $\widetilde{a}\in \re$. Let $a>|\widetilde{a}|$ and let $\varphi$ be a non-negative, compactly supported function on $N$ such that $\varphi|_K=a$. Then the function
	\begin{equation*}
		f(z):=\int_G\varphi(gz)\,dg,
	\end{equation*}
	is $G$-invariant, non-negative and cocompactly supported. Also, there holds $A+f\geq \widetilde{a}+a>0$ on $Y$, by construction, and $A+f\geq c>0$ on $N\backslash Y$, by assumption. We conclude that $A+f\geq r:= \min\{\widetilde{a}+a,c\}>0$ on the whole space $N$.
	\end{proof}

\begin{lemma}\label{positivepd}
	There exists a $G$-invariant, non-negative, smooth, cocompactly supported function $f\in C^{\infty}(N)^G$ and a constant $r>0$ such that $P^2+f\geq r>0$ on $\ee^0(S)$.
\end{lemma}
\begin{proof}
 The result follows immediately from the positivity of $\widetilde{\Delta}$ on  $\ee^0(S)$ and Lemmas \ref{positivel2diracschro} and \ref{positivitynoncocomp}, since then $P^2+f\geq A +f\geq r>0$ on $\ee^0(S)$. 
\end{proof}
\begin{proposition}\label{propginvertatinf}
	The operator $P$ on $\ee^0(S)$ is $G$-invertible at infinity.
\end{proposition}
\begin{proof}
The result follows from Lemma \ref{positivepd} and by repeating the arguments in \cite[\S5.3]{haoguo}.
\end{proof}

\subsection{The Localised Equivariant Coarse Index}\label{subseclocaleqcoarind}

Let $S \to N$ be as at the start of this section.
 Suppose, for now, that $G = \Gamma$ is a finitely generated discrete group acting freely on $N$. 
Recall Definition \ref{def D star}.
\begin{definition}
An operator $T  \in \B(L^2(S) \otimes l^2(\bN))$
is \emph{supported near} $K$ if there is an $R>0$   such that for all $f \in C_0(N)$ whose support is at least a distance $R$ from $K$, we have $f \circ T = T  \circ f = 0$.

The equivariant Roe algebra $C^*(N; K)^{\Gamma} \subset C^*(N)^{\Gamma}$ of $M$, localised at a $\Gamma$-invariant, closed subset $K \subset M$, is the closure in $\B(L^2(S) \otimes l^2(\bN))$ of the subalgebra of $\Gamma$-invariant, locally compact operators with finite propagation that are supported near $K$.
\end{definition}

We assume until the end of this subsection that $P$ is a $\Gamma$-equivariant Dirac-type operator on $S$, and that there exists a closed subset $K \subset N$ and $c>0$ such that, for all $s \in \Gamma_c^{\infty}(S)$ supported outside $K$, we have $\|Ps\|_{L^2(S)} \geq c \|s\|_{L^2(S)}$. 
By Lemmas 2.1 and 2.3 in \cite{Roe16}, there is a normalising function $\chi$ such that $\chi(P) \in D^*(N)^{\Gamma}$ and $\chi(P)^2-1 \in C^*(N; K)^{\Gamma}$.

Since $P$ is odd with respect to some $\Gamma$-invariant $\bZ/2\bZ$ grading on $S$, the operator $\chi(P)$ is also odd-graded, because $\chi$ is odd.
Let $\chi(P)_+$ be the restriction to even-graded sections. There is an isometry $U \colon L^2(S^-)\otimes l^2(\bN)\to L^2(S^-) \otimes l^2(\bN)$ that lies in $D^*(N)^{\Gamma}$; see \cite[Lemma 7.7]{HR95}.
Then the class of $U \circ \chi(P)_+ \otimes 1_{l^2(\bN)}$ in $D^*(N)^{\Gamma}/C^*(N; K)^{\Gamma}$ is invertible. Hence, this operator defines a class
\[
[P] \in K_1(D^*(N)^{\Gamma}/C^*(N; K)^{\Gamma}).
\]
Applying the boundary map in the six-term exact sequence associated to the ideal $C^*(N; K)^{\Gamma} \subset D^*(N)^{\Gamma}$, we obtain
\begeq{eq ind P Roe}
\partial[P] \in K_0(C^*(N; K)^{\Gamma}))  = K_0(C^*(K)^{\Gamma} ) = K_0(C^*_r(\Gamma)). 
\eneq
For the first equality of $K$-theory groups, see Lemma 1 in Section 5 of \cite{HRY93}. For the second, which relies on compactness of $K/\Gamma$, see 
 \cite[Theorem 5.3.2]{WillettYu}. 
\begin{definition}\label{def roe index}
The \emph{localised equivariant coarse index} of $P$,
\[
\Index_{\Gamma}^{\loc}(P) \in K_0(C^*_r(\Gamma)),
\]
 is the class \eqref{eq ind P Roe}.
\end{definition}

We return to the more general setting where $G$ is not necessarily discrete, and its action on $N$ is not necessarily free.
Definition \ref{def roe index} was generalised to this setting, in Definition 3.4 in \cite{guohochsvarg}. The versions of the algebras $D^*(N)^G$ and $C^*(N; K)^G$ used in this definition are slightly different from the versions used in the case of free actions by discrete groups. We refer to Sections 2 and 3 of \cite{guohochsvarg} for details.
This construction yields an index
\begin{equation}\label{eqloceqcoarseind}
	\Index_G^{\text{loc}}(P)\in K_0(\CS_r(G)). 
\end{equation}

We now note that the three notions of equivariant index for $P$ coincide, namely, the Fredholm index (\ref{eqindexdef}), the localised $G$-index (\ref{eq loc g ind}) and the localised equivariant coarse index (\ref{eqloceqcoarseind}), whenever they exist. Indeed, by \cite[Proposition 6.5]{guohochsvarg}, if $P$ is Fredholm on $\ee^0(S)$, then its Fredholm index coincides with its localised equivariant coarse index:

	\begin{equation}\label{eqfredindeqlocind}
	\Index^{\rm{loc}}_G(P)=\Index(P)\in K_0(\CS_r(G)).
	\end{equation}
For the remaining equality, we adapt a proof of the $G$-Callias operator case to our setting, as follows:
\begin{proposition}\label{guo2021}
Suppose $P$ is $G$-invertible at infinity. Then
\begin{equation*}\Index_G^{\rm{loc}}(P)= \Ind_G(P)\in K_0(\CS_r(G)).	
\end{equation*}
\end{proposition}
\begin{proof}
Following the proof of  \cite[Proposition 6.6]{Guo_2021}, let $b\in C_b(\re)$ be the function
\[ b(x):= \begin{cases} 
     -1 & x\leq -1, \\
     x & 0\leq x \leq 1, \\
     1 & x\geq 1,
   \end{cases}
\]
and, for each $s>0$, let $b_s\in C_b(\re)$ be defined by
\begin{equation*}
	b_s(x):=\frac{x}{(|x|^{1/s}+1)^s}.
\end{equation*}
It is immediate to see that $\lim_{s\to 0}\|b_s-b\|_{\infty}=0$. Define the operator
\begin{equation*}
	P_s:= b_{s/2}(P)
\end{equation*}
on $\ee^0(S)$. We note that
\begin{equation}\label{eqpsext}
	P_s=\frac{P}{(P^2+f)^{1/2}}(P^2+f)^{1/2}\psi_{s/2}(P),
\end{equation}
where $f$ comes from the definition of $G$-invertibility at infinity, and
\begin{equation*}
	\psi_s:=\left(x\mapsto \frac{1}{(|x|^{1/s}+1)^s}\right )\in C_0(\re).
\end{equation*}
The operator $P(P^2+f)^{-1/2}$ is invertible modulo $\mathcal{K}(\ee^0(S))$ by \cite[Proposition 4.13]{haoguo}, and thus so is $P_s$ for each $s>0$. In particular, for $s=1$, we see that the operator $P_1$ defines a class $[P_1]\in K_1(\mathcal{L}(\ee^0(S))/\mathcal{K}(\ee^0(S)))$. Since $(P^2+f)^{1/2}\psi_{s/2}(P)$ is invertible, one can see by (\ref{eqpsext}) that $\Ind_G(P)=\partial[P_1]\in K_0(\CS_r(G))$. Now, because $\lim_{s\to 0} \|P_s\oplus 1-b(P)\oplus 1\|=0$, it follows that $P_s$ is a continuous path of operators which are invertible modulo $\mathcal{K}(\ee^0(S))$, connecting $b(P)\oplus 1$ to $P_1$. This homotopy establishes the desired result.
\end{proof}

\subsection{The Spin Dirac Operator}\label{subsecspindiracop}

\begin{proof}[Proof of Proposition \ref{prop Spin SFG}]
The family \eqref{eq Dt Spin} satisfies conditions \ref{assump1} by assumption, and \ref{assump3} by the Lichnerowicz' formula. The map $t\mapsto D_t$ has uniformly bounded weak derivative (cf. \cite[\S 4.3.3]{van_den_Dungen_2016}). Hence, the family is norm-continuous by Remark \ref{diffimpliescont}. Since the family $\{D_t\}_{t\in\re}$ is locally constant outside of $K=[0,1]$, we see that the Assumptions \ref{assump1}-\ref{assump4} are thus satisfied. 
\end{proof}

Until the end of this subsection, we consider the setting of Subsection \ref{sec rho eta invar}. Form the exterior tensor product $S_0\boxtimes \com$ over $M\times \re$ and trivially extend the $G$-action to $\re$. By acting as the identity on the fibers in the direction of $\com\times\re\to\re$, we can see that $S_0\boxtimes \com\to M\times \re$ defines a $G$-equivariant Hermitian vector bundle. We consider the $\CS_r(G)$-Hilbert module $\ee^0(M\times \re, (S_0\boxtimes \com)^{\oplus 2})$ by completing the space of sections $\Gamma_c^{\infty}(M\times \re,(S_0\boxtimes \com)^{\oplus 2})$ with respect to the $\ee^0$-norm.  The product operator
\begin{equation*}
	\widetilde{P}_{D}:=
	\begin{pmatrix}
	0&-i\partial_t+i\DD\\-i\partial_t-i\DD&0
\end{pmatrix}
\end{equation*}
can be regarded as an operator on $\ee^0(M\times \re,(S_0\boxtimes \com)^{\oplus 2})\simeq (\ee^0(E)\otimes L^2(X,F))^{\oplus 2}$. We write
\begin{equation*}
	-i\widetilde{\partial_t}:=\begin{pmatrix}
		0&-i\partial_t\\-i\partial_t&0
	\end{pmatrix},\quad\quad 
	\widetilde{D}_{\bullet}:=\begin{pmatrix}
		0& i\DD\\
		-i\DD&0
	\end{pmatrix}.
\end{equation*}
Because of (\ref{eq weitzenbock}), we may write 
\begin{equation*}
	D_t^2=\Delta_t+\kappa_t,
\end{equation*}
where $\Delta_t:=(\tau_t^{-1}\circ(\nabla^{S_t})^*\nabla^{S_t}\circ\tau_t)$ and $\kappa_t:=\tau^{-1}_t\circ\slashed{\kappa}_t\circ\tau_t$.
Because the family $\{D_t\}_{t\in\re}$ is constant outside of $[0,1]$, the operators $-i\partial_t$ and $D_t$ commute for every $t\in \re\backslash[0,1]$. Hence, we may write \begin{equation}\label{eqsquareop}
 \widetilde{P}_{D}^2=(-i\widetilde{\partial}_t)^2+\widetilde{D}_{\bullet}^2=(-i\widetilde{\partial}_t)^2+\widetilde{\Delta}_\bullet+\widetilde{\kappa}_{\bullet},
 \end{equation}	
 outside of $M\times[0,1]$, where
	\begin{equation*}
	\widetilde{\Delta}_{\bullet}:=\begin{pmatrix}
		\Delta_{\bullet}&0\\ 0& \Delta_{\bullet}
	\end{pmatrix},\quad\quad 
	\widetilde{\kappa}_{\bullet}:=\begin{pmatrix}
		\kappa_{\bullet}&0\\
		0& \kappa_{\bullet}
	\end{pmatrix}. 
\end{equation*}
We note immediately that $(-i\widetilde{\partial}_t)^2 +\widetilde{\Delta}_\bullet$ is a $G$-equivariant differential operator which is positive on $\ee^0(M\times X,(E\boxtimes F)^{\oplus 2})$, and that $\widetilde{\kappa}_\bullet$ is a $G$-equivariant endomorphism of $(S_0\boxtimes \com)^{\oplus 2}$ such that there exists a constant $c$ such that $\widetilde{\kappa}_\bullet\geq c>0$ outside of the cocompact set $M\times K$. This is the case since $\widetilde{\kappa}_t$ is equal to either $\widetilde{\kappa}_0$ or $\widetilde{\kappa}_1$ for all $t\notin [0,1]$, and by the assumption that $g_0$ and $g_1$ have positive scalar curvature. It follows from Proposition \ref{propginvertatinf}, and the discussion above it, that $\widetilde{P}_{ D}$ is $G$-invertible at infinity. It thus defines a localised $G$-index
\begin{equation*}
	\Ind_G(\widetilde{P}_{ D})\in K_0(\CS_r(G)).
\end{equation*}
On the level of $K\!K$-theory, this index is given by the class 
\begin{equation*}
	[(\ee^0(M\times X,(S_0\boxtimes \com)^{\oplus 2}), \widetilde{P}_{D}(\widetilde{P}_{ D}^2+f)^{-1/2},1_{\com})]\in K\!K^0(\com,\CS(G)),
\end{equation*}
 where the function $f$ on $M\times X$ comes from Lemma \ref{positivepd}.

Let now $c\colon T^*(M\times \re)\to \End(\widetilde{S})$ and $\nabla^{\widetilde{S}}$ denote, respectively, the Clifford action and spin connection on the spinor bundle $\widetilde{S}$ over $M\times\re$. The spin Dirac operator is then defined by
\begin{equation*}
	\slashed{D}_{M\times\re}:=c\circ \nabla^{\widetilde{S}}.
\end{equation*}
 Using the isometries $\tau_t$ as in (\ref{eqisometry}) we can form an isometry
 \begin{equation}\label{eqtau}
 	\tau\colon (S_0\boxtimes \com)^{\oplus 2}\to \widetilde{S}
 \end{equation}
     over $M\times \re$. Indeed, fiberwisely we see that
\begin{equation*}
	(S_0\boxtimes\com)^{\oplus 2}_{(m,t)}=((S_0)_m\times\{t\})^{\oplus 2}\xrightarrow{\tau_t\times 1} ((S_t)_m\times \{t\})^{\oplus 2} =\widetilde{S}_{(m,t)},
\end{equation*}
where the first isomorphism is given by (\ref{eqisometry}) and the last equality by (\ref{eqrestriction}). Using the isomorphism $(\ee^0(M,S_0)\otimes L^2(\re,\com))^{\oplus 2}\simeq \ee^0(M\times \re,(S_0\boxtimes\com)^{\oplus 2})$, and by conjugating with the map (\ref{eqtau}), we can perceive both $\slashed{D}_{M\times\re}$ and $\widetilde{P}_{ D}$ as operators on $\ee^0(M\times \re,(S_0\boxtimes\com)^{\oplus 2})$. We write 
$$\widetilde{P}:=\tau^{-1}\circ \slashed{D}_{M\times\re}\circ\tau$$
 in what follows.

\begin{lemma}\label{lemmadifops}
 The difference
	\begin{equation*}
		A:=\widetilde{P}-\widetilde{P}_{D}
	\end{equation*}
	is a $G$-equivariant, cocompactly supported endomorphism of $(S_0\boxtimes \com)^{\oplus 2}$.
\end{lemma}
\begin{proof}
	We note that outside of the cocompact set $M\times[0,1]$ the two operators coincide, since $\slashed{D}_{M\times\re}$ is of product form therein. In \cite[\S3]{bgm}, an expression for $\slashed{D}_{M\times\re}$ around the submanifold $M_t=M\times\{t\}$ is computed in terms of the spin Dirac operator $\slashed{D}_{M_t}$ on $M_t$ (see (\ref{eq spin dirac op mt})), and a choice of vector field $\nu$ on $M\times \re$ such that $\langle \nu,\nu\rangle =1$ and $\langle \nu, TM\rangle=0$. Choose $\nu=-\partial_t$, and represent the Clifford action by $\nu$ on $\widetilde{S}\big{|}_{M_t}=S_t\oplus S_t$ via
	\begin{equation*}
		c(\nu)=\begin{pmatrix}
			0 &i \\ i &0
		\end{pmatrix}.
	\end{equation*}
We see by \cite[Equation (11)]{bgm} that
\begin{equation*}
	\slashed{D}_{M\times\re}=\begin{pmatrix}
		0&-i\nabla_{\partial_t}^{\widetilde{S}}+i\slashed{D}_{M_t}\\-i\nabla_{\partial_t}^{\widetilde{S}}-i\slashed{D}_{M_t}&0
	\end{pmatrix} + R
\end{equation*}
around $M_t$, where $R\in \End(\widetilde{S})$ is an endomorphism depending on the mean curvature. The result follows from noting that the principal symbols of $\widetilde{P}$ and $\widetilde{P}_{D}$ are the same. Indeed, by picking a smooth function $f\in C^{\infty}(M\times\re)$ we can immediately calculate that $\tau^{-1}\circ [\slashed{D}_{M_t},f]\circ\tau=[D_t,f]$ and $[\nabla_{\partial_t}^{\widetilde{S}},f]=[\partial_t,f]=f'\in C^{\infty}(M\times\re)$, as desired.

\end{proof}
\begin{lemma}\label{lemmaequalityofindex}
The operator $\slashed{D}_{M\times\re}$ is $G$-invertible at infinity and
\begin{equation*}
\Ind_\Gamma(\widetilde{P}_{D})=\Ind_{\Gamma}(\slashed{D}_{M\times\re})\in K_0(\CS_r(G)).
	\end{equation*}
\end{lemma}
\begin{proof}

	By Lichnerowicz' formula and the methods in Subsection \ref{subseclocalgind}, we see that the operator $\slashed{D}_{M\times\re}$ and, hence, $\widetilde{P}$ are $\Gamma$-invertible at infinity. Moreover, we can find a $\Gamma$-invariant, non-negative, smooth, cocompactly supported function $f\in C^{\infty}(M\times\re)$ as in Lemma \ref{positivepd} such that both $\Ind_\Gamma(\widetilde{P}_{D})$ and $\Ind_\Gamma(\widetilde{P})$ are represented in $K\!K^0(\com,\CS_r(\Gamma))$, respectively, by the operators $\widetilde{P}_{D}(\widetilde{P}_{D}^2+f)^{-1/2}$ and $\widetilde{P}(\widetilde{P}^2+f)^{-1/2}$. Using Lemma \ref{lemmadifops}, we compute
	\begin{multline}\label{eqcompactdifference}
			\widetilde{P}_{D}(\widetilde{P}_{D}^2+f)^{-1/2}-\widetilde{P}(\widetilde{P}^2+f)^{-1/2}=\widetilde{P}_{D}(\widetilde{P}_{D}^2+f)^{-1/2}-(\widetilde{P}_{D}+A)((\widetilde{P}_{D}+A)^2+f)^{-1/2}\\
			=\widetilde{P}_{D}((\widetilde{P}_{D}+f)^{-1/2}-((\widetilde{P}_{D}+A)^2+f)^{-1/2})
			-A((\widetilde{P}_{D}+A)^2+f)^{-1/2}.
	\end{multline}
	The term $-A((\widetilde{P}_{D}+A)^2+f)^{-1/2}$ is in $\mathcal{K}(\ee^0(\widetilde{M}\times\re,(S_0\boxtimes\com)^{\oplus 2}))$, which follows from Lemma \ref{lemmadifops} and by the equivariant version of Rellich Lemma (Proposition \ref{lemmaeqrellich}). In order to treat the remaining term, we use a resolvent identity: for every $\lambda\geq 0$, there holds
\begin{multline}\label{eqresolvident}
			(\widetilde{P}_{D}+f+\lambda^2)^{-1}-((\widetilde{P}_{D}+A)^2+f+\lambda^2)^{-1} =\\ (\widetilde{P}_{D}+f+\lambda^2)^{-1}((\widetilde{P}_{D}+A)^2-
\widetilde{P}_{D})((\widetilde{P}_{D}+A)^2+f+\lambda^2)^{-1}.
	\end{multline}
	The first term on the right-hand side of (\ref{eqcompactdifference}) can be computed by using the integral form (\ref{eqintegralformula}). Using identity (\ref{eqresolvident}), we see that it is equal to
	\begin{equation}\label{eqintform}
		\frac{2}{\pi}\int_0^{\infty}\widetilde{P}_{D}(\widetilde{P}_{D}+f+\lambda^2)^{-1}(\widetilde{P}_{D}A+A\widetilde{P}_{D}+A^2)((\widetilde{P}_{D}+A)^2+f+\lambda^2)^{-1}\,d\lambda.
	\end{equation}
The result follows by noting that the three integrands in (\ref{eqintform}), corresponding to the three summands in the middle factor, are also compact operators.
\end{proof}

\begin{proposition}\label{prop SF APS}
In the setting of Subsection \ref{sec rho eta invar},
\[
\Sf_{\Gamma}(\DD)  = \Index_\Gamma^{\rm{loc}}(\slashed{D}_{M\times\re}) \in K_0(C^*_r(\Gamma))
\]

\end{proposition}
\begin{proof}
 By  Proposition \ref{prop Spin SFG}, Theorem \ref{propfredindeqsfg} and the comments above it, 
  the operator\begin{equation}\label{eqdiracsch}
	\widetilde{P}_{D}=\begin{pmatrix}0&-i\partial_t+i\DD\\-i\partial_t-i\DD&0
    \end{pmatrix}
\end{equation}
is regular self-adjoint and Fredholm on $(\ee^0(S_0)\otimes L^2(\re,\com))^{\oplus 2}$, and
\begin{equation}\label{eqsfeqind}
	\Sf_\Gamma(\DD) = \Index(\widetilde{P}_{D}).
\end{equation}
So the claim follows from
 \eqref{eqfredindeqlocind}, Proposition \ref{guo2021} and Lemma \ref{lemmaequalityofindex}.
\end{proof}

\subsection{Higher APS-index Theorems}\label{subsecdelocrhoeta}

Theorems \ref{thm eta} and \ref{theoremmain} now follow from Proposition \ref{prop SF APS} and higher APS-index theorems.

\begin{proof}[Proof of Theorem \ref{thm eta}]
In the setting of this theorem, several higher APS-index theorems imply existence of the algebra $\A(G)$ in \eqref{eq tauh K}, convergence of the relevant delocalised $\eta$-invariants, and the equality
\begeq{eq APS eta}
\tau_h(\Index_G^{\rm{loc}}(\slashed{D}_{M\times\re} ) )= \int_{(M \times [0,1])^h} f_h \frac{\hat A((M \times [0,1])^h)}{\det(1-hR^{\mathcal{N}})^{1/2}} -\frac{1}{2}(\eta_h(D_1) - \eta_h(D_0)).
\eneq
Here $f_h$ is a cutoff function for the action by $Z$ on $(M \times [0,1])^h$, and $R^{\mathcal{N}}$ is the curvature of the connection on the normal bundle $\mathcal{N} \to (M \times [0,1])^h$ to $(M \times [0,1])^h$ in $M \times [0,1]$ induced by the Levi-Civita connection. More precisely, we apply
\begin{itemize}
\item \cite[Theorem 5.3]{xieyu} in the case where $G$ is discrete and finitely generated, the conjugacy class of $h$ has polynomial growth, and $G$ acts freely on $M$;
\item \cite[Corollary 2.10]{HWW22} in the case where $G$ is discrete and finitely generated, the conjugacy class of $h$ has polynomial growth, and $G$ has slow enough exponential growth;
\item \cite[Theorem 4.39]{PPST21} in the case where $G$ is a connected, real semisimple Lie group and $g$ is semisimple; and 
\item \cite[Corollary 2.10]{HWW22} in the case where $h=e$.
\end{itemize}
Then one notes that by multiplicativity of the $\hat A$-form and the fact that $\hat A(\mathbb{R}) = 1$, 
\[
\hat A((M \times [0,1])^h) = \hat A(M^h \times [0,1]) = p^*\hat A(M^h), 
\]
where $p\colon M^h \times [0,1] \to M^h$ is projection onto the first factor. Letting $N\to  M^h$ be the normal bundle to $M^h$ in $M$, we similarly have
\begin{equation*}
	\det(1-hR^{\mathcal{N}})=\det(1-hR^N),
\end{equation*}
where $R^N$ is the curvature of the induced connection on $N$. Thus, we obtain the equality
\[
\tau_h(\Index_G^{\rm{loc}}(\slashed{D}_{M\times\re} ) )= \int_{M^h} f_h \frac{p^*\hat A(M^h)}{\det(1-hR^{N})^{1/2}} -\frac{1}{2}(\eta_h(D_1) - \eta_h(D_0)).
\]
The result now follows from Proposition \ref{prop SF APS}.
\end{proof}

\begin{proof}[Proof of Theorem \ref{theoremmain}]
By \cite[Theorem 1.14]{piazzaschick}, we have
\begeq{eq APS rho}
\iota_*\bigl(\Index_{\Gamma}^{\rm{loc}}(\slashed{D}_{M\times\re}) \bigr) =  (j_1)_*(\rho(g_1)) - (j_0)_*(\rho(g_0)).
\eneq
The result follows from \eqref{eq APS rho} and Proposition \ref{prop SF APS}.
\end{proof}
\begin{remark}
See \cite[Theorem A]{xieyu2} for a version of \eqref{eq APS rho} that includes the case of odd-dimensional manifolds. That result is stated in terms of Yu's localisation algebras, but the link with the formulation involving the algebra $D^*(M)^{\Gamma}$ is given in  \cite[Section 6]{xieyu2}.
\end{remark}

\bibliographystyle{alpha}
\bibliography{references.bib}

\end{document}